\documentclass[11pt]{amsart}
\usepackage{amssymb,amstext,amsmath,amscd,amsthm,amsfonts,enumerate,graphicx,latexsym,accents}
\usepackage[usenames]{color}
\usepackage[all,cmtip]{xy}
\usepackage[colorlinks=true, citecolor=blue, linkcolor=red,hyperindex]{hyperref}

\newtheorem{thm}{Theorem}[section]
\newtheorem{cor}[thm]{Corollary}
\newtheorem{lem}[thm]{Lemma}
\newtheorem{prop}[thm]{Proposition}
\theoremstyle{definition}
\newtheorem{dfn}[thm]{Definition}
\newtheorem{rem}[thm]{Remark}
\newtheorem{ques}[thm]{Question}

\newtheorem{ex}[thm]{Example}

\newtheorem*{claim*}{Claim}
\theoremstyle{remark}

\numberwithin{equation}{thm}
\def\Hom{\operatorname{Hom}}
\def\tr{\operatorname{tr}}

\def\ann{\operatorname{Ann}}
\def\Tor{\operatorname{Tor}}
\def\Ext{\operatorname{Ext}}

\def\syz{\Omega}
\def\cm{\mathrm{CM}}
\def\mod{\operatorname{mod}}
\def\im{\operatorname{Im}}
\def\ker{\operatorname{Ker}}
\def\cok{\operatorname{Coker}}

\def\id{\mathrm{id}}

\def\codim{\operatorname{codim}}
\def\pd{\operatorname{pd}}
\def\edim{\operatorname{edim}}
\def\ng{\operatorname{NG}}
\def\p{\mathfrak{p}}
\def\m{\mathfrak{m}}

\def\h{\operatorname{H}}
\def\V{\operatorname{V}}
\def\E{\operatorname{E}}
\def\n{\mathfrak{n}}
\def\I{\operatorname{I}}

\def\depth{\operatorname{depth}}
\def\t{\mathrm{t}}
\def\soc{\operatorname{Soc}}
\def\xx{\boldsymbol{x}}

\def\res{\operatorname{res}}
\def\X{\mathcal{X}}
\def\speco{\operatorname{Spec^0}}
\def\sing{\operatorname{Sing}}
\def\D{\mathrm{D}}
\def\sg{\mathrm{sg}}
\def\b{\mathrm{b}}
\def\nf{\operatorname{NF}}
\def\rest{\operatorname{rest}}
\def\rank{\operatorname{rank}}
\def\thick{\operatorname{thick}}
\def\ipd{\operatorname{IPD}}
\def\spec{\operatorname{Spec}}
\def\P{\mathbb{P}}
\def\codepth{\operatorname{codepth}}

\def\C{\mathcal{C}}
\def\k{\mathrm{K}}
\begin{document}
\title{Burch ideals and Burch rings}
\author{Hailong Dao}
\address[Hailong Dao]{Department of Mathematics, University of Kansas, Lawrence, KS 66045-7523, USA}
\email{hdao@ku.edu}
\urladdr{https://www.math.ku.edu/~hdao/}
\author{Toshinori Kobayashi}
\address[Toshinori Kobayashi]{Graduate School of Mathematics, Nagoya University, Furocho, Chikusaku, Nagoya 464-8602, Japan}
\email{m16021z@math.nagoya-u.ac.jp}
\author{Ryo Takahashi}
\address[Ryo Takahashi]{Graduate School of Mathematics, Nagoya University, Furocho, Chikusaku, Nagoya 464-8602, Japan\,/\,Department of Mathematics, University of Kansas, Lawrence, KS 66045-7523, USA}
\email{takahashi@math.nagoya-u.ac.jp}
\urladdr{http://www.math.nagoya-u.ac.jp/~takahashi/}
\dedicatory{Dedicated to Lindsay Burch}
\begin{abstract}
We introduce the notion of Burch ideals and Burch rings.
They are easy to define, and can be viewed as generalization of many well-known concepts, for example integrally closed ideals of finite colength and Cohen--Macaulay rings of minimal multiplicity.
We give several characterizations of these objects.
We show that they satisfy many interesting and desirable properties: ideal-theoretic, homological, categorical.
We relate them to other classes of ideals and rings in the literature. 
\end{abstract}
\maketitle
\section{Introduction}

This article introduces and studies a class of ideals and their affiliated rings which we call Burch ideals and Burch rings.
While their definitions are quite simple, our investigation shows that they enjoy remarkable ideal-theoretic and homological properties.
These properties allow us to link them to many classes of ideals and rings in the literature, and consequently strengthen numerous old results as well as establish new ones.

Let us make a brief remark on our motivation and historical context.
The project originated from our effort to understand a beautiful result by Burch on homological properties of ideals below (\cite[Theorem 5(ii) and Corollary 1(ii)]{B}).

\begin{thm}[Burch]\label{re1}
Let $(R,\m)$ be a local ring.
Let $I$ be an ideal of $R$ with $\m I\not=\m(I:\m)$.
\begin{enumerate}[\rm(1)]
\item
Let $M$ be a finitely generated $R$-module.
If $\Tor_n^R(R/I,M)=\Tor_{n+1}^R(R/I,M)=0$ for some positive integer $n$, then $M$ has projective dimension at most $n$.
\item
If $I$ has finite projective dimension, then $R$ is regular.
\end{enumerate}
\end{thm}

Lindsay Burch\footnote{We are grateful to Rodney Sharp and Edmund Robertson for providing us with the following brief biography of Burch: Lindsay Burch was born in 1939. She did her first degree at Girton College, Cambridge from 1958 to 1961. She then went to Exeter University to study for a Ph.D. advised by David Rees.  She was appointed to Queen's College, Dundee in 1964 before the award of her Ph.D. which wasn't until 1967 for her thesis ``Homological algebra in local rings". At the time she was appointed to Queen's College it was a college of the University of St Andrews but later, in 1967, it became a separate university, the University of Dundee. Burch continued to work in the Mathematics Department of the University of Dundee until at least 1978.  She then took up computing and moved to a computing position at Keele University near Stafford in the north of England. She remained there until she retired and she still lives near Keele University.} was a PhD student of David Rees, and she wrote several (short) papers that have had a sizable impact on two active corners of commutative algebra: homological theory and integral closure of ideals. Perhaps most researchers in the field know of her work via the frequently used Hilbert--Burch Theorem (\cite{B}), her construction of ideals with only three-generators while possessing arbitrarily complicated homological behavior (\cite{B1}), and the Burch inequality on analytic spreads (\cite{B2}). The ideas of Burch's particular result above, while less well-known, have resurfaced in the work of several authors which also motivated our work, see \cite{CGHPU, CHKV, KS, KV, SV}. However, it has appeared to us that what was known previously is just the tip of an iceberg, and led us to formally make the following definitions. 

Let $(R,\m)$ be a local ring.
We define an ideal $I$ of $R$ to be a {\em Burch ideal} if $\m I\not=\m(I:\m)$.
We also define {\em Burch rings of depth zero} to be those local rings whose completions are quotients of regular local rings by Burch ideals.
Then we further define {\em Burch rings} of positive depth as local rings which ``deform'' to Burch rings of depth zero; see Section \ref{basicSec} for the precise definitions.

It is not hard to see that the class of Burch ideals contains other well-studied classes: integrally closed ideals of codepth zero (under mild conditions), $\m$-full ideals, weakly $\m$-full ideals, etc.  

One of our main results characterizes Burch ideals and Burch rings of depth zero:

\begin{thm}[Theorem \ref{210}]\label{t11}
Let $(R,\m,k)$ be a local ring and $I\not=\m$ an ideal of $R$.
Then $I$ is Burch if and only if the second syzygy $\syz^2_{R/I} k$ of $k$ over $R/I$ contains $k$ as a direct summand.
\end{thm}

From this, we can quickly deduce a characterization of Gorenstein Burch ideals, which extends results on integrally closed or $\m$-full ideals in \cite{G, GH}. In fact, our proofs allow us to completely characterized modules over Burch rings
of depth zero whose some higher syzygies contain the residue field as a direct summand, as follows:

\begin{thm}[Theorem \ref{11}]\label{100}
Let $(R, \m,k)$ be a Burch ring of depth zero.
Let $M$ be a finitely generated $R$-module.
The following are equivalent:
\begin{enumerate}[\rm(1)]
\item
The ideal $\I(M)$ generated by all entries of the matrices $\partial_i$, $i>0$ in a minimal free resolution $(F,\partial)$ of $M$ is equal to $\m$. 
\item
The $R$-module $k$ is a direct summand of $\syz_R^rM$ for some $r\ge2$.
\end{enumerate}
\end{thm}

Our work reveals some interesting connections between Burch ideals/rings and concepts studied by other authors in quite different contexts.
For instance, we show that in codimension two, artinian almost Gorenstein rings as 
introduced by Huneke--Vraciu \cite{HV} (also studied in \cite{SV}) are Burch; see Proposition \ref{c68}.
Over a regular local ring, the ``Burchness'' of an ideal $I$ imposes a strong condition on the matrix at the end of a minimal free resolution of $I$, a condition that also appeared  in the work of Corso--Goto--Huneke--Polini--Ulrich  \cite{CGHPU} on iterated socles.
That connection led us to obtain a refinement of their result in Theorem \ref{t63}.  

We also study Burch rings of higher depth, especially their homological and categorical aspects.
We completely classify Burch rings which are fibre products in Proposition \ref{fp}.
The Cohen--Macaulay rings of minimal multiplicity are Burch.
Non-Gorenstein Burch rings turn out to be {\em G-regular} in Theorem \ref{33}, in the sense that all the totally reflexive modules are free.
Moreover, we show an explicit result on vanishing behavior of $\Tor$ for any pair of modules. 

\begin{thm}[Corollary \ref{1}]
Let $R$ be a Burch ring of depth $t$.
Let $M,N$ be finitely generated $R$-modules.
Assume that there exists an integer $l\ge\max\{3,t+1\}$ such that $\Tor_i^R(M,N)=0$ for all $l+t\le i\le l+2t+1$.
Then either $M$ or $N$ has finite projective dimension.
\end{thm}

To state our last main result in this introduction, recall that the {\em singularity category} $\D_\sg(R)$ is by definition the triangulated category given as the Verdier quotient of the bounded derived category of finitely generated $R$-modules by perfect complexes.
Under some assumptions, one can classify all the thick subcategories of $\D_\sg(R)$ for a Burch ring $R$.

\begin{thm}[Theorem \ref{th54}]\label{101}
Let $R$ be a singular Cohen--Macaulay Burch ring.
Suppose that on the punctured spectrum $R$ is either locally a hypersurface or locally has minimal multiplicity.
Then there is a one-to-one correspondence between the thick subcategories of $\D_{\sg}(R)$ and the specialization-closed subsets of $\sing R$.
\end{thm}

Next we describe the structure of the paper as well as other notable results.
In Section \ref{basicSec} we state our convention, basic definitions and preliminary results. 
Section \ref{cyclicSec} is devoted to giving a sufficient condition for a module to have a second syzygy having a cyclic direct summand (Proposition \ref{p1}). This is a generalization of \cite[Lemma 4.1]{KV}, and has an application to provide an exact pair of zero divisors (Corollary \ref{c26}).
These materials are used in Section \ref{main1Sec} and are perhaps of independent interest. 

In Section \ref{main2Sec}, we focus on the study of Burch rings of positive depth.
We verify that the class of Gorenstein Burch rings coincides with that of hypersurfaces (Proposition \ref{r5}).
Cohen--Macaulay local rings of minimal multiplicity with infinite residue field are Burch (Proposition \ref{minmulti}). Quotients of polynomial rings by perfect ideals with linear resolution are Burch (Proposition \ref{detEx}).
We also consider the subtle question of whether the Burch property is preserved by cutting down by {\em any} regular sequence consisting of minimal generators of $\m$.
Remarkably, this holds for Cohen--Macaulay local rings of dimension one with minimal multiplicity (Proposition \ref{r13}).
However, the answer turns out to be negative in general (Example \ref{e44}).

In Section \ref{main3Sec} we focus more deeply on Burch ideals in a regular local ring.
We gave a complete characterization in dimension two and link Burch rings and Burch ideals to various other concepts. 
Moreover, we give a characterization of the Burch local rings $(R,\m,k)$ with $\m^3=0$ in terms of a Betti number of $k$, the embedding dimension and type of $R$ (Theorem \ref{cu}).
We also characterize the Burch monomial ideals of regular local rings (Proposition \ref{4}).

In Section \ref{main4Sec}, we explore the homological and categorical aspects of Burch rings.
We find out the significant property of Burch rings that every module of infinite projective dimension contains a high syzygy of the residue field in its resolving closure (Proposition \ref{r1}).
We apply this and make an analogous argument as in \cite{NT} to classify various subcategories.

\section{Convention, definitions and basic properties of Burch ideals and rings}\label{basicSec}

Throughout this paper, we assume that all rings are commutative and noetherian, that all modules are finitely generated and that all subcategories are full and strict.
For a local ring $(R,\m,k)$, we denote by $\edim R$ the embedding dimension of $R$, by $r(R)$ the (Cohen--Macaulay) type of $R$, and by $\k^R$ the Koszul complex of $R$, i.e., the Koszul complex of a minimal system of generators of $\m$.
We set $\k^R=0$ when $R$ is a field.
For an $R$-module $M$, we denote by $\ell_R(M)$ the length of $M$, by $\mu_R(M)$ the minimal number of generators of $M$, and by $\beta_i^R(M)$ the $i$th Betti number of $M$.
The $i$th syzygy of $M$ in the minimal free resolution of $M$ is denoted by $\syz_R^iM$.
We omit subscripts and superscripts if there is no fear of confusion.


The remaining of this section deals with the formal notion of Burch ideals and Burch rings and their basic properties.

\begin{dfn}
Let $(R,\m)$ be a local ring.
We define a {\em Burch ideal} as an ideal $I$ with $\m I\not =\m(I:_R \m)$.
Note by definition that any Burch ideal $I$ of $R$ satisfies $\depth R/I=0$.
\end{dfn}

Here are some quick examples of Burch ideals.
Many more examples will follow from our results later. 

\begin{ex}\label{r3}
\begin{enumerate}[(1)]
\item
Let $(R,xR)$ be a discrete valuation ring.
Then $(x^n)$ is a Burch ideal of $R$ for all $n\ge1$, since $x(x^n)=(x^{n+1})\ne(x^n)=x(x^{n-1})=x((x^n):(x))$.
\item
Let $I$ be an ideal of a local ring $(R,\m)$.
Put $J=\m I$ and suppose $J\ne0$.
Then $\m(J:\m)=J\not=\m J$, so $J$ is a Burch ideal of $R$.
\item 
By the previous item, if $(R,\m)$ has positive depth then $I=\m^t$ is Burch for  any $t\geq 1$. More generally, if $\m^{t+1} \subseteq I\subseteq \m^{t}$, then $I$ is Burch if and only if $I:\m\neq \m^t$ and $I\m \neq \m^{t+1}$. Using this one can show that the set of Burch ideals is Zariski-open in $\text{Grass}_k(r, \m^t/\m^{t+1})$, for each $r = \dim_k I/\m^{t+1}$. 
\item
Let $(R,\m)$ be a local ring of positive depth.
Let $I$ be an integrally closed ideal of $R$.
Then $\m I:\m = I$ by the determinantal trick, so it is Burch.
See Proposition \ref{23} below. 
\end{enumerate}
\end{ex}

The following proposition gives some basic characterizations of Burch ideals.

\begin{prop} \label{23}
Let $(R,\m)$ be a local ring and $I$ an ideal of $R$.
The following are equivalent.\\
{\rm(1)} $I$ is a Burch ideal.\qquad
{\rm(2)} $(I:\m)\ne(\m I: \m)$.\qquad
{\rm(3)} $\soc(R/I)\cdot \m/I\m\not=0$.\\
{\rm(4)} $\depth R/I=0$ and $r(R/\m I)\not=r(R/I)+\mu(I)$.\\
{\rm(5)} $I\widehat{R}$ is a Burch ideal of $\widehat{R}$, where $\widehat{R}$ is the completion of $R$.
\end{prop}

\begin{proof}
(1) $\Leftrightarrow$ (2):
If $(I:\m)=(\m I:\m)$, then $\m(I:\m)=\m(\m I:\m)=\m I$.
Conversely, if $\m I=\m(I:\m)$, then $(\m I:\m)=(\m(I:\m):\m)=(I:\m)$.

(1) $\Leftrightarrow$ (3):
As $\soc R/I=(I:\m)/I$, we have $\soc R/I\cdot \m/I\m=0$ if and only if $\m(I:\m)=\m I$.

(2) $\Leftrightarrow$ (4): There are inclusions $\m I\subseteq I \subseteq (\m I:\m)\subseteq(I:\m)$, which especially says that $(\m I:\m)\ne(I:\m)$ implies $\depth R/I=0$.
We have $\ell((I:\m)/\m I)=\ell((I:\m)/I)+\ell(I/\m I)=r(R/I)+\mu(I)$ if $\depth R/I=0$, and $\ell((\m I:\m)/\m I)=r(R/\m I)$.
Thus, under the assumption $\depth R/I=0$, the equalities $(I:\m)=(\m I :\m)$ and $r(R/\m I)=r(R/I)+\mu(I)$ are equivalent.

(1) $\Leftrightarrow$ (5): It is clear that $\m I =\m(I:_R \m)$ if and only if $\widehat{\m}I=\widehat{\m}(I:_{\widehat{R}} \widehat{\m})$.
\end{proof}

Recall that an ideal $I$ of a local ring $(R,\m)$ is {\em $\m$-full} (resp.\ {\em weakly $\m$-full}) if $(\m I:x)=I$ for some $x\in\m$ (resp.\ $(\m I:\m)=I$).
Clearly, every $\m$-full ideal is weakly $\m$-full. The notion of $\m$-full ideals has been studied by many authors so far; see \cite{CNR,G,GH,W,W2} for instance.
Notably, it is fundamental to figure out the connections between $\m$-full ideals and another class of ideals.
For example, $\m$-primary integrally closed ideals are $\m$-full or equal to the nilradical of $R$ under the assumption that the residue field $k$ is infinite; see \cite[Theorem (2.4)]{G}.
There are many related classes of ideals, such as ideals satisfying the Rees property, contracted ideals and basically full ideals.
See \cite{HLNR,R} for the hierarchy of these classes.
The notion of weakly $\m$-full ideals is introduced in \cite[Definition 3.7]{CIST}.
The class of weakly $\m$-full ideals coincide with that of basically full ideals if they are $\m$-primary; see \cite[Theorem 2.12]{HRR}.
The following corollary is immediate from the implication (2) $\Rightarrow$ (1) in the above proposition.

\begin{cor}\label{2}
Let $(R,\m)$ be a local ring.
Let $I$ be an ideal of $R$ such that $\depth R/I=0$.
If $I$ is weakly $\m$-full, then it is Burch.
\end{cor}

Let $f:(S,\n,k)\to(R,\m,k)$ be a surjective homomorphism of local rings, and set $I=\ker f$.
Choi \cite{Ch} defines the invariant
$$
c_R(S,f)=\dim_k(\n(I:_S\n)/\n I).
$$
Clearly, an ideal $I$ of a local ring $(S,\n)$ is Burch if and only if Choi's invariant $c_{S/I}(S,\pi)$ is positive, where $\pi$ is the canonical surjection $S\to S/I$.
We give a description of Choi's invariant for a regular local ring.

\begin{prop} \label{p25}
Let $(R,\m,k)$ be a local ring, $(S,\n,k)$ a regular local ring, and $f:S\to R$ a surjective homomorphism with kernel $I$.
Then 
\begin{align*}
c_R(S,f)=
\begin{cases}
\dim_k \soc R+\dim_k \h_1(\k^R)-\edim R-\dim_k \h_1(\k^{R'})+\edim R' & (\text{if }I\ne\n),\\
\dim_k\n/\n^2 & (\text{if }I=\n),
\end{cases}
\end{align*}
where $R'=R/\soc R$.
\end{prop}

\begin{proof} 
Put $J=(I:_S \n)$.
We may assume $I\not=\n$, and hence $J\not=S$.
Then there are equalities
\begin{align*}
c_R(S,f)=\dim_k\n J/\n I
&=\ell(J/I)+(\ell(I/\n I)-\ell(\n/\n^2))-(\ell(J/\n J)-\ell(\n/\n^2))\\
&=\dim_k \soc R+(\dim_k\h_1(\k^R)-\edim R)-(\dim_k\h_1(\k^{R'})-\edim R').
\end{align*}
Now the proof of the proposition is completed.
\end{proof}

The above result especially says that in the case where $I\ne\n$ the number $c_R(S,f)$ is determined by the target $R$ of the surjection $f$.
Thus the following result is immediately obtained.

\begin{cor}[cf.\,{\cite[Theorem 2.4]{Ch}}]\label{5}
Let $R$ be a local ring that is not a field.
Let $(S_1,\n_1)$ and $(S_2,\n_2)$ be regular local rings, and $f_i:S_i\to R$ surjective homomorphisms for $i=1,2$.
Then the equality $c_R(S_1,f_1)=c_R(S_2,f_2)$ holds.
In particular, $\ker f_1$ is Burch if and only if so is $\ker f_2$.
\end{cor}

We are now ready to define Burch rings.

\begin{dfn} \label{d310}
Let $(R,\m)$ be a local ring of depth $t$.
Denote by $\widehat{R}$ the $\m$-adic completion of $R$.
We say that $R$ is {\em Burch} if there exist a maximal $\widehat{R}$-regular sequence $\xx=x_1,\dots,x_t$ in $\widehat{R}$, a regular local ring $S$ and a Burch ideal $I$ of $S$ such that $\widehat{R}/(\xx)\cong S/I$.
\end{dfn}

\begin{rem} \label{r312}
If $I$ is a Burch ideal of a local ring $(R,\m)$, then $R/I$ is a Burch ring of depth zero.
Indeed, $I\widehat{R}$ is a Burch ideal of $\widehat{R}$ by Proposition \ref{23}.
Take a Cohen presentation $\widehat R\cong S/J$, where $(S,\n)$ is a regular local ring.
Let $I'$ be the ideal of $S$ such that $I'\supseteq J$ and $I'/J=I\widehat{R}$.
Then one can easily verify that $\n I'\not=\n(I':_S \n)$, that is, $I'$ is a Burch ideal of $S$.
Note that the completion of the local ring $R/I$ is isomorphic to $S/I'$.
Hence $R/I$ is a Burch ring of depth zero.
\end{rem}

Let $R$ be a local ring.
The {\em codimension} and {\em codepth} of $R$ are defined by
$$
\codim R=\edim R-\dim R,\qquad
\codepth R=\edim R-\depth R.
$$
Then $R$ is said to be a {\em hypersurface} if $\codepth R\le1$.
This is equivalent to saying that the completion $\widehat{R}$ of $R$ is isomorphic to $S/(f)$ for some regular local ring $S$ and some element $f\in S$.

\begin{ex} \label{e310}
If $R$ is a hypersurface, then it is a Burch ring.
Indeed, take a regular sequence $\xx$ in $\widehat{R}$ such that $\widehat{R}/(\xx)$ is an artinian local ring with $\edim \widehat{R}/(\xx)\le 1$.
Then $\widehat{R}/(\xx)$ is isomorphic to the quotient ring of a discrete valuation ring $S$ by a nonzero ideal $I$.
By Example \ref{r3}(1), the ideal $I$ of $S$ is Burch.
\end{ex}

We define the invariant $c_R$ of a local ring $(R,\m,k)$ by
$$
c_R=\dim_k\soc R+\dim_k\h_1(\k^R)-\edim R-\dim_k\h_1(\k^{R'})+\edim R'.
$$
Here, we set $R'=R/\soc R$, and adopt the convention that $\dim_k \h_1(\k^{R'})=0=\edim R'$ in the case where $R'=0$ (i.e. $R$ is a field).
Then we can characterize the Burch rings of depth zero:

\begin{lem} \label{l310}
Let $(R,\m,k)$ be a local ring.
Then $c_R=c_{\widehat{R}}$, and the following are equivalent.

{\rm(1)} $R$ is a Burch ring and $\depth R=0$.\qquad{\rm(3)} $c_R\not=0$.

{\rm(2)} $\widehat{R}$ is a Burch ring and $\depth R=0$.\qquad{\rm(4)} $c_R>0$.\\
Moreover, if $R$ is a Burch ring of depth zero and isomorphic to $S/I$ for some regular local ring $(S,\n)$ and some ideal $I$ of $S$, then $I$ is a Burch ideal of $S$ or $I=\n$.
\end{lem}

\begin{proof}
The numbers $\dim_k \soc R,\,\dim_k \h_1(\k^R),\,\edim R,\,\dim_k \h_1(\k^{R'}),\edim R'$ are preserved by the completion of $R$.
In particular, one has $c_R=c_{\widehat{R}}$.
Furthermore, take a Cohen presentation $\widehat{R}\cong S/I$ with a complete regular local ring $S$.
Letting $\pi:S\to S/I$ be the natural surjection, we have $c_{\widehat{R}}=c_R(S,\pi)$.
This especially shows that $c_R$ is nonnegative.
Now we show the equivalence of (1)--(4).
It is obvious that (1) and (3) are equivalent to (2) and (4), respectively.
The equivalence of (2) and (3) follows from Proposition \ref{p25}.
Finally, we show the last assertion.
Suppose that $R$ is Burch of depth zero and that $R\cong S/I$, where $S$ is a regular local ring and $I$ is an ideal of $S$.
Then $\widehat R\cong T/J$ for some regular local ring $T$ and a Burch ideal $J$ of $T$.
There are surjections from the regular local rings $\widehat S$ (the completion of $S$) and $T$ to the local ring $\widehat S/I\widehat S\cong\widehat R\cong T/J$, and the kernel of the latter is the Burch ideal $J$.
Corollary \ref{5} implies that $I\widehat S$ is a Burch ideal of $\widehat S$, and $I$ is a Burch ideal of $S$ by Proposition \ref{23}.
\end{proof}

We end this section by proving an useful characterization of Burch ideals when $\depth R>1$. The only if direction is known for $\m$-full ideals; see \cite[Corollary 7]{W2}. 

\begin{lem} \label{31}
Let $(R,\m)$ be a local ring of depth $>1$.
An ideal $I$ of $R$ is Burch if and only if there exists a non-zerodivisor $a\in \m$ such that $R/\m$ is a direct summand of the $R$-module $I/aI$.
\end{lem}

\begin{proof}
Assume that $I$ is Burch.
Then there exist $a\in \m$ and $b\in(I:_R \m)$ such that $ab\in I\setminus \m I$.
We have $a\not\in\m^2$, since otherwise $ab\in \m^2(I:_R\m)=\m I$.
As $b\m\subseteq I$, it holds that $ab\m\subseteq aI$.
We can define an $R$-homomorphism $f:R/\m\to I/aI$ by $f(\overline1)=\overline{ab}$.
As $ab\not\in \m I$, the element $\overline{ab}$ is a part of a minimal system of generators of $I/aI$, and hence $f$ is a split monomorphism.

Conversely, assume that there is a split monomorphism $f:R/\m\to I/aI$, where $a\in R$ is a non-zerodivisor.
Let $c\in I$ be the preimage of $f(\overline1)\in I/aI$.
Then $c\m\subseteq aI\subseteq (a)$.
The assumption $\depth R>1$ implies $\depth R/(a)>0$.
Hence $c$ has to be in $(a)$, that is, there exists $b\in R$ with $c=ab$.
Observe $ab\m=c\m\subseteq aI$.
Then $a$ being non-zerodivisor yields $b\m\in I$.
In other words, $b\in(I:_R \m)$.
The image of $ab=c$ is a part of a minimal system of generators of $I/aI$, and we have $ab\not\in \m I$.
Thus $\m(I:_R \m)\not=\m I$, which means that $I$ is a Burch ideal.
\end{proof}

\begin{rem}
It is worth noting that Lemma \ref{31} can be used to give a quick proof of Theorem \ref{re1} when $\depth R>1$ and $n>1$. Namely, if $\Tor_n^R(R/I,M)=\Tor_{n+1}^R(R/I,M)=0$ then it follows that $\Tor_n^R(I/aI, M)=0$, which
implies that $\Tor_n^R(k, M)=0$.
\end{rem}

\section{Cyclic direct summands of second syzygies}\label{cyclicSec}

The main purpose of this section is to study sufficient conditions for an $R$-module to have a cyclic direct summand in its second syzygy. They will be used in the proofs of Section \ref{main1Sec} and are perhaps of independent interest. In fact, some of our proofs were motivated by the work of Kustin-Vraciu (\cite{KV}) and Striuli-Vraciu (\cite{SV}) which focused on different but related problems. 

We start by some simple criteria for a homomorphism $f:R\to M$ to be a split monomorphism.

\begin{lem} \label{l21}
Let $(R,\m)$ be a local ring of depth zero.
Let $f:R \to M$ be a homomorphism of $R$-modules.
Assume one of the following conditions holds.

{\rm(a)} $R$ is Gorenstein.\quad
{\rm(b)} $M$ is free.\quad
{\rm(c)} $M$ is a syzygy (i.e., a submodule of a free module).\\
Then the followings are equivalent.

{\rm(1)} $f$ is a split monomorphism.\qquad
{\rm(2)} $f$ is a monomorphism.\qquad
{\rm(3)} $f(\soc R)\not=0$.
\end{lem}

\begin{proof}
The implications (1) $\Rightarrow$ (2) $\Rightarrow$ (3) are clear.
To show (3) $\Rightarrow$ (1), put $C=\cok f$.

(a) As $R$ is Gorenstein, we have $\soc R\cong R/\m$.
The equality $f(\soc R)\not=0$ implies $\ker f\cap \soc R=0$.
Hence $\ker f=0$, and $f$ is injective.
As $\Ext^1_R(C,R)=0$, the map $f$ is split injective.

(b) If $f$ is not split injective, then $\im f$ is contained in $\m M$ by the assumption that $M$ is free.
This yields that the inclusions $\ker f \supseteq \ann(\m M) \supseteq \soc R$ hold.

(c) Let $g:M\to F$ be a monomorphism with $F$ free.
The composition $gf:R \to F$ satisfies $gf(\soc R)\not=0$.
By the previous argument, $gf$ is split injective.
There is a retraction $r:F\to R$ with $rgf=\text{id}_R$.
We see that $rg:M\to R$ is a retraction of $f$.
Therefore $f$ is split injective.
\end{proof}

Next we consider $R$-homomorphisms from a cyclic $R$-module to an $R$-module.

\begin{lem} \label{l22}
Let $R$ be a ring, $I$ an ideal of $R$ and $M$ an $R$-module.
Consider an $R$-homomorphism $f:R/I \to M$.
Then $f$ is split injective if and only if the composition map $pf:R/I \to M/IM$ is split injective, where $p:M\to M/IM$ is the natural surjection.
\end{lem}

\begin{proof}
Suppose $f$ is split injective.
Then there is an $R$-homomorphism $g:M \to R/I$ such that $gf=\text{id}_{R/I}$.
On the other hand, $g$ factor through $p:M \to M/IM$, that is $g=g'p$ for some $g':M/IM \to R/I$.
So we see that $g'$ is a retraction of $pf$.
Next, suppose $pf$ is split injective.
Then there is an $R$-homomorphism $h:R/I \to M/IM$ such that $hpf=\text{id}_{R/I}$.
Thus $hp:M \to R/I$ is a retraction of $f$.
\end{proof}

For a matrix $A$ over $R$ we denote by $\I_i(A)$ the ideal of $R$ generated by the $i$-minors of $A$.
For a linear map $\phi$ of free $R$-modules, we define $\I_i(\phi)$ as the ideal $\I_i(A)$, where $A$ is a presentation matrix of $\phi$.
The following lemma is well-known; we state it for the convenience of the reader.

\begin{lem} \label{l23}
Let $R^n \xrightarrow{d} R^m \to M \to 0$ be exact.
If $\I_1(d)\subseteq I$, then $M/IM$ is $R/I$-free.
\end{lem}

\begin{proof}
The tensored sequence $(R/I)^n\xrightarrow{d\otimes R/I} (R/I)^m \to M/IM \to 0$ is exact.
Since $\I_1(d)$ is contained in $I$, we see that $d\otimes R/I=0$, and hence $M\cong(R/I)^m$.
\end{proof}

We generalize \cite[Lemma 4.1]{KV} as follows.

\begin{prop} \label{p1}
Let $(S,\n,k)$ be a local ring and $I\subseteq J$ ideals of $S$.
Set $R=S/I$.
Let $\cdots \to R^q \xrightarrow[]{\overline{C}}R^p \xrightarrow[]{\overline{B}} R^n \xrightarrow[]{\overline{A}} R^m \to M \to 0$ be a minimal $R$-free resolution of an $R$-module $M$, where $A,B,C,\dots$ are matrices over $S$.
Assume that $J$ satisfies either of the following conditions.\\
\qquad
{\rm(a)} $J\supseteq \I_1(A)+\I_1(C)$.\qquad
{\rm(b)} $J\supseteq \I_1(A)$ and $S/J$ is Gorenstein.\\
If $(I:_S J)\not\subseteq (IJ:_S (J:_S \n)\I_1(A))$, then $S/J$ is a direct summand of $\syz^2_R M$.
\end{prop}

\begin{proof}
For each integer $i$, let $J_i$ be the ideal of $S$ generated by the entries of the $i$th column of $A$.
Then $\I_1(A)=J_1+\cdots+J_n$, and $(I:_S J)\not\subseteq (IJ:_S (J:_S \n)\I_1(A))=(IJ:_S (J:_S \n)J_1)\cap\cdots\cap(IJ:_S (J:_S \n)J_n)$.
Hence $(I:_S J)\not\subseteq(IJ:_S (J:_S \n)J_s)$ for some $s$.
Choose an element $u\in(I:_S J)\setminus (IJ:_S (J:_S \n)J_s)$ and let $v\in R^n$ be the image of $u\cdot e_s$, where $e_s$ is the $s$th unit vector of $S^n$.
Since $Ju\subseteq I$ and $\I_1(A)\subseteq J$, $v$ is in $\ker \overline{A}=\syz^2_RM=:X$.
We can define an $R$-homomorphism $f:S/J\to X$ by $f(\overline{1})=v$.

Now we want to show $f$ is split injective.
By Lemma \ref{l22}, it is enough to verify so is the induced map $f'=pf:S/J\to X/\overline{J}X$.
By Lemmas \ref{l21} and \ref{l23}, it suffices to check $f'(\soc S/J)\not=0$.

Since $u\not\in ((IJ):_S (J:_S \n)J_s)$, we can choose an element $a\in(J:_S \n)$ such that $auJ_s\not\subseteq IJ$.
Remark that $a\not\in J$, otherwise one has $au\in I$, which forces $auJ_S$ to be contained in $IJ$.
Let $\overline{a}$ be the image of $a$ in $S/J$.
We have that $0\not=\overline{a}\in \soc S/J$.
If $f'(\overline{a})=0$, then $av\in \overline{J}X$.
Then there exist elements $x\in JR^p$ and $y\in IR^n$ such that $aue_s=Bx+y$.
Observe that $auAe_s=ABx+Ay\in IJR^m$.
So we obtain the inclusion $auJ_s\subseteq IJ$, which is contradiction.
Thus $f'(\overline{a})=0$ and we conclude that $f$ is split injective.
\end{proof}

As a corollary, we have the following restatement of \cite[Lemma 4.1]{KV}.

\begin{cor}\label{10}
Let $(S,\n,k)$ be a local ring and $I$ an ideal of $S$.
Set $R=S/I$ and consider a minimal $R$-free presentation $R^n \xrightarrow{\overline A} R^m \to M \to0$ of an $R$-module $M$, where $A$ is an $m\times n$ matrix over $S$ and $\overline A$ is the corresponding matrix over $R$.
If $(I:_S\n)\nsubseteq(\n I:_S\I_1(A))$, then $k$ is a direct summand of $\syz_R^2M$.
\end{cor}

Recall that a module $M$ over a ring $R$ is called {\em totally reflexive} if the natural map $M\to M^{\ast\ast}$ is an isomorphism and $\Ext_R^i(M,R)=\Ext_R^i(M^\ast,R)=0$ for all $i>0$, where $(-)^\ast=\Hom_R(-,R)$.
Over a Cohen--Macaulay local ring, a totally reflexive module is a maximal Cohen--Macaulay module, and the converse holds as well over a Gorenstein local ring.

Also, recall that a pair $(x,y)$ of elements of a ring $R$ is called an {\em exact pair of zerodivisors} if the equalities $(0:_Rx)=yR$ and $(0:_Ry)=xR$ hold.
This is equivalent to saying that the sequence $\cdots\xrightarrow{x}R\xrightarrow{y}R\xrightarrow{x}R\xrightarrow{y}\cdots$ is exact.
It is easy to see that for each exact pair of zerodivisors $(x,y)$ the $R$-modules $R/xR$ and $R/yR$ are totally reflexive.

The following result is another application of Proposition \ref{p1}.

\begin{cor} \label{c26}
Let $(S,\n,k)$ be a local ring and $I\subseteq J$ ideals of $S$.
Assume that $S/I,S/J$ are Gorenstein and that $(I:_S J)\nsubseteq (IJ):_S ((J:_S \n)J)$.
Then there exist elements $a,b\in S$ such that $J=I+(a)$, $(I:_SJ)=I+(b)$, and $(\overline{a},\overline{b})$ is an exact pair of zerodivisors of $S/I$.
\end{cor}

\begin{proof}
Put $R=S/I$.
Consider a minimal $R$-free resolution $\cdots \to R^n\xrightarrow[]{\overline{A}} R \to S/J \to 0$ of the $R$-module $S/J$.
Clearly, the equality $\I_1(A)+I=J$ holds.
We can derive from Proposition \ref{p1} that the $R$-module $\syz_R^2 (S/J)$ has a direct summand isomorphic to $S/J$.
Since $R$ is Gorenstein and the $R$-module $S/J$ is indecomposable, $\syz_R^2 (S/J)$ is also indecomposable.
This implies that $\syz_R^2 (S/J)\cong S/J$, that is, the sequence $0 \to S/J \to R^n \to R \to S/J \to 0$ is exact.
We have $\ell(R^n)+\ell(S/J)=\ell(R)+\ell(S/J)$, which yields $n=1$.
Thus the ideal $J/I$ of $R$ is principal, and we find $a\in R$ with $J/I=aR$.
As $(0:_Ra)=\syz^1_R (J/I)\cong S/J$, the ideal $(0:_Ra)$ of $R$ is also principal.
Taking a generator $b$ of $(0:_Ra)$, we get an exact pair of zerodivisors $(a,b)$ of $R$.
\end{proof}


\section{Proof of Theorem \ref{210} and some applications}\label{main1Sec}

This section concerns with a surprising characterization of Burch rings of depth zero below, and some applications. 

\begin{thm} \label{210}
Let $(R,\m,k)$ be a local ring that is not a field.
Then $R$ is a Burch ring of depth zero if and only if $k$ is isomorphic to a direct summand of its second syzygy $\syz^2_R k$. 
\end{thm}

We shall delay the proof until the end of this section. First, note that we can interpret Corollary \ref{10} in terms of Burch rings as follows.
Here we use the notation $\I_1(M)$ for an $R$-module $M$ to be the ideal $\I_1(A)$ where $A$ is a matrix in a minimal free presentation $F\xrightarrow[]{A} G\to M \to 0$ of $M$.
Remark that $\I_1(M)$ is independent of the choice of $A$ (see \cite[page 21]{BH} for instance).

\begin{prop}\label{203}
Let $(R,\m,k)$ be a Burch ring of depth zero that is not a field.
Let $M$ be an $R$-module with $\I_1(M)=\m$.
Then $k$ is a direct summand of $\syz_R^2M$.
In particular, $k$ is a direct summand of $\syz_R^2k$.
\end{prop}

\begin{proof}
By \cite[Corollary 1.15]{LW}, the module $\syz_R^2M$ contains $k$ as a direct summand if and only if so does $\syz_R^2M\otimes_R \widehat{R}\cong\syz_{\widehat{R}}^2(M\otimes_R\widehat{R})$.
Hence we may assume that $R$ is complete, and then there is a regular local ring $(S,\n)$ and a Burch ideal $I\subset \n^2$ such that $R\cong S/I$.
Consider a minimal $R$-free presentation $R^n \xrightarrow{\overline A} R^m \to M \to0$ of an $R$-module $M$, where $A$ is a matrix over $S$ and $\overline A$ is $A$ modulo $I$.
Then we see that $\I_1(\overline A)=\I_1(M)=\m$, which implies that $\I_1(A)=\n$.
Hence $(I:_S\n)\not\subseteq (\n I:_S\I_1(A))$, and thus $k$ is a direct summand of $\syz_R^2M$ by Corollary \ref{10}.
\end{proof}

Here is an immediate consequence of the above proposition.

\begin{cor}
Let $(R,\m,k)$ be an artinian Burch ring.
Then there exists an element $x\in\m\setminus\m^2$ such that $k$ is a direct summand of the ideal $(0:_Rx)$ of $R$.
\end{cor}

\begin{proof}
Let $x_1,\dots,x_n$ be a minimal system of generators of $\m$.
There is an exact sequence
$$
0 \to \bigoplus_{i=1}^n(0:x_i) \to R^n \xrightarrow{\partial} R^n \to \bigoplus_{i=1}^nR/(x_i) \to 0\quad\text{with}\quad\partial=\left(\begin{smallmatrix}x_1&&&\\&x_2&&\\&&\dots&\\&&&x_n\end{smallmatrix}\right).
$$
This shows $\I_1(\partial)=\m$ and $\syz^2(\bigoplus_{i=1}^nR/(x_i))=\bigoplus_{i=1}^n(0:x_i)$.
Proposition \ref{203} implies that $k$ is a direct summand of $\bigoplus_{i=1}^n(0:x_i)$.
Since $R$ is artinian, it is henselian.
The Krull--Schmidt theorem shows that $k$ is a direct summand of $(0:x_i)$ for some $i$.
\end{proof}

The following theorem classifies $\m$-primary Gorenstein Burch ideals.

\begin{thm}\label{r9}
Let $(R,\m)$ be a local ring and $I$ an $\m$-primary ideal.
The following are equivalent.\\
{\rm(1)} $I$ is a Burch ideal of $R$ and $R/I$ is Gorenstein.\\
{\rm(2)} $I$ is weakly $\m$-full and $R/I$ is Gorenstein.\qquad
{\rm(3)} $I$ is $\m$-full and $R/I$ is Gorenstein.\\
{\rm(4)} $I=(x_1^r,x_2,\dots,x_n)$ with $x_1,\dots,x_n$ a minimal system of generators of $\m$ and $n,r>0$.
\end{thm}

\begin{proof}
It follows from \cite[Proposition (2.4)]{GH} that (3) is equivalent to (4), while it is obvious that (3) implies (2) and (2) implies (1).
Assume (1) to deduce (4).
Remark \ref{r312} shows that $R/I$ is a Burch ring.
Proposition \ref{203} implies that $k$ is a direct summand of $\syz_{R/I}^2k$.
As $\syz_{R/I}^2k$ is indecomposable, we get $k\cong\syz_{R/I}^2k$, whence $R/I$ is a hypersurface.
Thus $\m/I$ is cyclic.
Choose an element $x_1\in\m$ such that $\overline{x_1}$ is a minimal generator of $\m/I$.
Then $x_1$ is a minimal generator of $\m$, and $\m=I+(x_1)$.
There is a unique integer $r>0$ with $x_1^r\in I$ and $x_1^{r-1}\notin I$.
Choose $x_2,\dots,x_n\in I$ so that $\overline{x_2},\dots,\overline{x_n}$ is a minimal system of generators of $I(R/(x_1))=\m/(x_1)$.
We see that $x_1,x_2,\dots,x_n$ is a minimal system of generators of $\m$.
Clearly, $I$ contains $J:=(x_2,\dots,x_n)$.
Note that every $\m/J$-primary ideal is a power of $\m/J=((x_1)+J)/J$.
As $x_1^r\in I$ and $x_1^{r-1}\notin I$, we get $I/J=((x_1^r)+J)/J$.
This shows $I=(x_1^r,x_2,\dots,x_n)$.
\end{proof}

We now characterize the modules over a Burch ring having the residue field as a direct summand of some high syzygy.

\begin{thm}\label{11}
Let $(R,\m,k)$ be a Burch local ring of depth zero which is not a field.
Let $M$ be an $R$-module.
Take a minimal free resolution $(F,\partial)$ of $M$.
The following are equivalent.\\
\quad
{\rm(1)} One has $\sum_{i>0}\I_1(\partial_i)=\m$.\qquad
{\rm(2)} $k$ is a direct summand of $\syz_R^rM$ for some $r\ge2$.\\
In particular, if $\sum_{i>0}\I_1(\partial_i)=\m$, then there exists an integer $i\ge3$ such that $\I_1(\partial_i)=\m$.
\end{thm}

\begin{proof}
(2) $\Rightarrow$ (1):
The minimal presentation matrix $A$ of $\syz_R^rM$ is equivalent to $\left(\begin{smallmatrix}
B&0\\
0&C
\end{smallmatrix}\right)$, where $B$ and $C$ are the minimal presentation matrices of $k$ and $N$, respectively.
Hence $\I_1(\partial_{r+1})=\I_1(A)=\I_1(B)+\I_1(C)=\m+\I_1(C)=\m$, which shows $\sum_{i>0}\I_1(\partial_i)=\m$.

(1) $\Rightarrow$ (2):
We may assume that $R$ is complete, and hence there is a regular local ring $(S,\n)$ and a Burch ideal $I\subseteq S$ with $R\cong S/I$.
For each $i>0$ we identify $\partial_i$ with a matrix over $R$, and let $d_i$ be a matrix over $S$ lifting $\partial_i$.
Then $\n=\sum_{i>0}\I_1(d_i)+I$.
The noetherian property shows $\n=\I_1(d_1)+\cdots+\I_1(d_n)+I$ for some $n>0$.
Hence $(\n I:\n)=(\n I:\I_1(d_1)+\cdots+\I_1(d_n)+I)=(\n I:\I_1(d_1))\cap\cdots\cap(\n I:\I_1(d_n))\cap(\n I:I)$.
Since $I$ is Burch, we have $(I:\n)\nsubseteq(\n I:\n)$ by Proposition \ref{23}.
In particular $I$ is nonzero, and we see that $(I:\n)\subseteq\n=(\n I:I)$.
We obtain $(I:\n)\nsubseteq(\n I:\I_1(d_t))$ for some $1\le t\le n$.
It follows from Corollary \ref{10} that $k$ is a direct summand of the cokernel of $\partial_t$, which is $\syz_R^{t+1}M$.
\end{proof}

We can now complete the proof of Theorem \ref{210}. 

\begin{proof}[Proof of Theorem \ref{210}]
We first show the ``only if'' part.
The module $\syz^2_R k$ contains $k$ as a direct summand if and only if $\soc \syz^2_R k\not\subseteq \m \syz^2_R k$.
Hence we may assume that $R$ is complete.
There are a regular local ring $S$ and a Burch ideal $I$ of $S$ such that $R\cong S/I$.
It follows by Proposition \ref{203} that $k$ is a direct summand of $\syz^2_R k$.

Now we consider the ``if'' part.
Again we may assume that $R$ is complete.
Take a Cohen presentation $R\cong S/I$, where $(S,\n)$ is a regular local ring and $I$ is an ideal of $S$ contained in $\n^2$.
If $(I:_S \n)\nsubseteq \n^2$, then there is an element
$x\in(\m\cap \soc R)\setminus\m^2$.
One has a decomposition $\m=J\oplus (x)$, which means that $R$ is of the form $S\times_k T$ with $\edim T=1$.
Then $R$ is Burch by Example \ref{e310} and Proposition \ref{lfp}.
Thus we may assume that $(I:_S \n)\subseteq \n^2$.
Suppose that $I$ is not Burch, so that $\n(I:_S\n)=\n I$. We aim to show that $\soc \syz^2_R k\subseteq \m \syz^2_R k$.
Take minimal generators $x_1,\dots,x_e$ of $\n$.
There is a commutative diagram
$$
\xymatrix{
&& 0 \ar[d] & 0 \ar[d] & \\
&& \syz^2_S k \ar[d] \ar[r] & \syz^2_R  \ar[d]k &\\
0 \ar[r] & I^e \ar[d] \ar[r] & S^e \ar[d] \ar[r] & R^e \ar[d] \ar[r] & 0 \\
0 \ar[r] & I \ar[d] \ar[r] & \n \ar[d] \ar[r] & \m \ar[d] \ar[r] & 0 \\
& I/\n I \ar[d] & 0 & 0 & \\
& 0 & &
}
$$
of $S$-modules with exact rows and columns.
Applying the snake lemma, we get an exact sequence 
\begin{equation} \label{eq}
\syz^2_S k \to \syz^2_R k \xrightarrow[]{\delta} I/\n I \to 0,
\end{equation}
where $\delta$ sends each element $a\in \syz^2_R k$ whose preimage in $S^e$ is ${}^\t(a_1,\dots,a_e)$ to the image of $\sum_i x_ia_i$ in $I/\n I$. 
Now consider element $a\in \soc \syz^2_R k$.
This means that the preimage ${}^\t(a_1,\dots,a_e)\in S^e$ of $a$ satisfies $a_i\in(I:_S\n)$ for all $i$.
Therefore, the element $\sum_i x_ia_i \in S$ is contained in $\n(I:_S\n)=\n I$.
This yields that $\delta(a)=0$.
By the exact sequence \eqref{eq}, we can take the preimage $(a_1,\dots,a_e)\in S^e$ of $a$ to be contained in $\syz^2_S k$.
We already have ${}^\t(a_1,\dots,a_e)\in (I:_S\n)S^e\subseteq \n^2S^e$.
It follows that ${}^\t(a_1,\dots,a_e)\in\syz^2_S k\cap\n^2S^e\subseteq \n\syz^2_S k$, see \cite[Theorems 3.7 and 4.1]{HSV} for the second containment.
Consequently, the element $a$ is contained in $\m\syz^2_R k$.  This allows us to conclude that if $\soc \syz^2_R k\not\subseteq \m \syz^2_R k$ then $I$ is a Burch ideal, and hence $R$ is a Burch ring.
\end{proof}

In view of Theorem \ref{210}, one may wonder if an artinian local ring $R$ is Burch if the residue field $k$ is a direct summand of $\syz^nk$ for some $n\ge3$.
This is not true in general:

\begin{ex}\label{r8}
Let $k$ be a field, and consider the ring $R=k[\![x,y]\!]/I$, where $I=(x^4,x^2y^2,y^4)$.
The minimal free resolution of $k$ is
$$
0 \gets k\gets R\xleftarrow{\left(\begin{smallmatrix}x&y\end{smallmatrix}\right)} R^2\xleftarrow{\left(\begin{smallmatrix}-y&xy^2&x^3&0\\x&0&0&y^3\end{smallmatrix}\right)}R^4\xleftarrow{\left(\begin{smallmatrix}xy^2&0&x^3&0&0&y^3&0&0\\y&x&0&0&0&0&y^2&0\\0&0&y&x&0&0&0&y^2\\0&0&0&0&y&-x&0&0\end{smallmatrix}\right)}R^8\gets\cdots.
$$
We have $\soc\syz^3k=\soc R^4=(x^3y,xy^3)R^4$.
The column vector $z:={}^\t(x^3y,0,0,0)=y\cdot{}^\t(x^3,0,y,0)-{}^\t(0,0,y^2,0)$ is in $\soc\syz^3k\setminus\m\syz^3k$.
The assignment $1\mapsto z$ makes a split monomorphism $k\to\syz^3k$, and $k$ is a direct summand of $\syz^3k$.
However, $R$ is not Burch as one can easily check the equality $\m(I:\m)=\m I$.
\end{ex}

\section{Burch rings of positive depth}\label{main2Sec}

In this section, we study Burch rings of positive depth.
First of all, let us investigate what Gorenstein Burch rings are.

\begin{prop}\label{r5}
A local ring is Burch and Gorenstein if and only if it is a hypersurface.
\end{prop}

\begin{proof}
Let $R$ be a local ring of dimension $d$.
If $R$ is hypersurface, then $R$ is clearly Gorenstein, and it is also Burch by Example \ref{e310}.
Conversely, suppose that $R$ is Burch and Gorenstein.
Then there exists a system of parameters $\xx=x_1,\dots,x_d$ such that $\widehat R/(\xx)$ is an artinian Gorenstein Burch local ring.
By definition, there exist a regular local ring $(S,\n)$ and a Burch ideal $I$ of $S$ such that $\widehat R/(\xx)\cong S/I$.
By Theorem \ref{r9}, there are a minimal system of generators $y_1,\dots,y_n$ of $\n$ with $n>0$ and an integer $r>0$ such that $I=(y_1^r,y_2,\dots,y_n)$.
In particular, $S/I\cong\widehat R/(\xx)$ is a hypersurface, and so is $R$.
\end{proof}

A Cohen--Macaulay local ring $R$ is said to have {\em minimal multiplicity} if $e(R)=\codim R+1$.

\begin{prop}\label{minmulti}
Let $(R,\m, k)$ be a  Cohen--Macaulay local ring with minimal multiplicity and infinite residue field $k$. Then $R$ is Burch. 
\end{prop}

\begin{proof}
We can find a general system of parameters $\underline x$ such that $A= R/(\underline x)$ is artinian and still has minimal multiplicity. This simply means that 
$\m_A^2=0$, so the first syzygy of $k$ is a $k$-vector space. Thus $A$ is Burch by \ref{210} and so is $R$. 
\end{proof}

We establish a lemma to prove our next result on Burch rings.

\begin{lem}\label{r12}
Let $(R,\m,k)$ be a $1$-dimensional Cohen--Macaulay local ring with minimal multiplicity.
Then there exists an isomorphism $\m^\ast\cong\m$, where $(-)^\ast=\Hom_R(-,R)$.
\end{lem}

\begin{proof}
If $R$ is a discrete valuation ring, then $\m\cong R$, and hence $\m^\ast\cong\m$.
So we assume that $R$ is not a discrete valuation ring.
Since $R$ has minimal multiplicity, by \cite[Lemma 1.11]{Lip}, there is an $R$-regular element $x\in\m$ such that $\m^2=x\m$.
Let $Q$ be the total quotient ring of $R$.
We have
$$
\m^\ast=\Hom_R(\m,R)\cong\Hom_R(\m,xR)\cong(xR:_Q\m)\supseteq\m,
$$
where the second isomorphism follows from \cite[Proposition 2.4(1)]{lp} for instance.
For each element $\frac{a}{s}\in(xR:_Q\m)$, we have $ax\in a\m\subseteq sxR$, which implies $a\in sR$ as $x$ is $R$-regular, and hence $\frac{a}{s}\in R$.
Therefore $(xR:_Q\m)$ is an ideal of $R$ containing $\m$.
Since $R$ is not a discrete valuation ring, it is a proper ideal.
We get $(xR:_Q\m)=\m$, and consequently $\m^\ast\cong\m$.
\end{proof}

Cohen--Macaulay rings of dimension $1$ with minimal multiplicity have a remarkable property.

\begin{prop}\label{r13}
Let $(R,\m,k)$ be a $1$-dimensional Cohen--Macaulay local ring with minimal multiplicity.
Then the quotient artinian ring $R/(x)$ is a Burch ring for any  parameter $x\in\m\setminus\m^2$.
\end{prop}

\begin{proof}
If $R$ is regular, then it is a discrete valuation ring, and $x$ is a uniformizer.
Hence $R/(x)$ is a field, and it is Burch.
Thus we assume that $R$ is singular.
Applying $(-)^\ast=\Hom_R(-,R)$ to the natural exact sequence $0\to\m\to R\to k\to0$, we get an exact sequence $0\to R\to\m^\ast\to k^{\oplus r}\to0$, where $r$ is the type of $R$.
Making the pullback diagram of the map $\m^\ast\to k^{\oplus r}$ and the natural surjection $R^{\oplus r}\to k^{\oplus r}$, we obtain an exact sequence $0 \to \m^{\oplus r} \to R^{\oplus(r+1)} \to \m^\ast\to0$.
As $R$ is singular, $\m^{\oplus r}$ does not have a nonzero free summand by \cite[Corollary 1.3]{D}.
We get an isomorphism $\m^{\oplus r}\cong\syz(\m^\ast)$.
Combining this with Lemma \ref{r12} yields $\m^{\oplus r}\cong\syz\m\cong\syz^2k$.
Since $x$ is an $R$-regular element in $\m\setminus\m^2$, there is a split exact sequence $0 \to k \to \m/x\m \to \m/(x) \to 0$, which induces $\m/x\m\cong k\oplus\m/(x)$.
We obtain isomorphisms of $R/(x)$-modules
\begin{align*}
k^{\oplus r}\oplus(\m/(x))^{\oplus r}
&\cong(\m/x\m)^{\oplus r}
\cong\syz^2k/x\syz^2k\\
&\cong\syz_{R/(x)}(\m/x\m)
\cong\syz_{R/(x)}k\oplus\syz_{R/(x)}(\m/(x))
\cong\syz_{R/(x)}k\oplus\syz_{R/(x)}^2k,
\end{align*}
where the third isomorphism holds since there is an exact sequence $0\to\syz^2k\to R^{\oplus n}\to\m\to0$ with $n=\edim R$, which induces an exact sequence $0\to\syz^2k/x\syz^2k\to(R/(x))^{\oplus n}\to\m/x\m\to0$.
As $R/(x)$ is an artinian local ring, it is henselian.
The Krull--Schmidt theorem implies that $k$ is a direct summand of either $\syz_{R/(x)}k$ or $\syz_{R/(x)}^2k$.
In the former case, applying $\syz_{R/(x)}(-)$ shows that $k$ is a direct summand of $\syz_{R/(x)}^2k$.
Theorem \ref{210} concludes that $R/(x)$ is a Burch ring.
\end{proof}

\begin{prop}\label{detEx}
Let $S=k[x_1,\dots,x_n]$ be a polynomial ring over an infinite field and $I\subset S$ is a homogenous ideal such that $S/I$ is Cohen-Macaulay and $I$ has a linear resolution. Then $R=(S/I)_{\m}$ is Burch where $\m= (x_1,\dots,x_n)$. 
\end{prop}

\begin{proof}
Let $A= S/I$ and $(l_1,\dots, l_d)$ be a general linear system of parameters on $A$. We write $A/(l_1,\dots, l_d)A$ as $T/J$ where $T$ is a polynomial ring in $n-d$ variables over $k$ and $J$ is a zero-dimensional ideal. Then $J$ still has linear resolution. Assume $I$ (and $J$) are generated in degree $t$, then the regularity of $J$ is $t$, but since $J$ is zero-dimensional, the socle degree of $J$ is $t-1$. Thus $J=\n^t$ where $\n$ is the irrelevant ideal of $T$, and so $R$ is Burch by  definition and Example \ref{r3}.
\end{proof}

\begin{ex}
There are many examples satisfying the conditions of Proposition \ref{detEx}. For example, let $m\geq n$ and let $I=I_n \subset k[x_{ij}]=S$ be the ideal generated by maximal minors in a $m$ by $n$ matrix of indeterminates. Then it is well-known that  $S/I$ is Cohen-Macaulay with $\dim S/I = (m+1)(n-1)$ and the $a$-invariant of $S/I$ is $-m(n-1)$ (see \cite{BH}). It follows that the regularity of $I$ is $n$, so it has linear resolution. 

Another source of examples are Stanley-Reisner rings of ``facet constructible" or ``stacked" simplicial complexes, see \cite[Theorem 4.1 and 4.4]{DS}. 
\end{ex}

We will show in Corollary \ref{r4} that if $x$ is a regular element of a local ring $(R,\m)$ such that $R/(x)$ is Burch, then $x\not\in\m^2$.
It is natural to ask whether the quotient ring $R/Q$ of a Burch ring $R$ is again Burch for any ideal $Q$ generated by regular sequence consisting of elements in $\m\setminus \m^2$.
This is true if $R$ is either a hypersurface or a Cohen--Macaulay local ring of dimension one with minimal multiplicity, as we saw in Propositions \ref{r5} and \ref{r13}.
The example below says that the question is not always affirmative.

\begin{ex} \label{e44}
Let $k$ be a field, and let $R=k[\![x,y,z]\!]/\I_2\!\left(\begin{smallmatrix}x^2&y&z^2\\y&z^2&x^2\end{smallmatrix}\right)$.
The Hilbert--Burch theorem implies that $R$ is a Cohen--Macaulay local ring of dimension $1$.
The ring $R$ is a Burch ring since so is the artinian quotient ring $R/(x)=k[\![y,z]\!]/(y^2,yz^2,z^4)$.
However, the artinian ring $R/(y)=k[\![x,z]\!]/(x^4,x^2y^2,y^4)$ is not Burch.
By Theorem \ref{210}, the $R$-module $k$ is a direct summand of $\syz_{R/(x)}^2k$, but not a direct summand of $\syz_{R/(y)}^2k$.
Incidentally, the module $k$ is a direct summand of $\syz_{R/(y)}^3k$ by Example \ref{r8}.
\end{ex}

To show our next result on Burch rings, we prepare a lemma on cancellation of free summands.

\begin{lem}\label{r11}
Let $R$ be a local ring.
Let $M,N$ be $R$-modules having no nonzero free summand.
If $M\oplus R^{\oplus a}\cong N\oplus R^{\oplus b}$ for some $a,b\ge0$, then $M\cong N$ and $a=b$.
\end{lem}

\begin{proof}
We may assume $a\ge b$.
Taking the completions, we get isomorphisms $\widehat{M}\oplus\widehat{R}^{\oplus a}\cong\widehat{N}\oplus\widehat{R}^{\oplus b}$.
Write $\widehat M=X\oplus\widehat{R}^{\oplus c}$ and $\widehat N=Y\oplus\widehat{R}^{\oplus d}$ with $c,d\ge0$ integers and $X,Y$ having no nonzero free summand.
Then $X\oplus\widehat{R}^{\oplus(c+a)}\cong Y\oplus\widehat{R}^{\oplus(d+b)}$.
As $\widehat R$ is henselian, we can apply the Krull-Schmidt theorem to deduce $X\cong Y$ and $c+a=d+b$.
Hence $d=c+(a-b)$, and we get $\widehat N
=Y\oplus\widehat{R}^{\oplus d}
\cong X\oplus\widehat{R}^{\oplus(c+(a-b))}
=\widehat M\oplus\widehat{R}^{\oplus(a-b)}
\cong\widehat{L}$, where $L:=M\oplus R^{\oplus(a-b)}$.
It follows from \cite[Corollary 1.15]{LW} that $N$ is isomorphic to $L$.
Since $N$ has no nonzero free summand, we must have $a=b$, and therefore $M=L\cong N$.
\end{proof}

The following result is a higher-dimensional version of the ``only if'' part of Theorem \ref{210}.

\begin{prop} \label{p46}
Let $(R,\m,k)$ be a singular Burch ring of depth $t$,
Then $\syz^tk$ is a direct summand of $\syz^{t+2}k$.
\end{prop}

\begin{proof}
We prove the proposition by induction on $t$.
The case $t=0$ follows from Lemma \ref{l310}, so let $t\ge1$.
There is an $R$-sequence $\xx=x_1,\dots,x_d$ such that $R/(\xx)$ is a Burch ring of depth zero.
Hence $R/(x_1)$ is a Burch ring of dimension $d-1$.
The induction hypothesis implies that $\syz_{R/(x_1)}^{t-1}k$ is a direct summand of $\syz_{R/(x_1)}^{t+1}k$.
Taking the syzygy over $R$, we see that $\syz_R\syz_{R/(x_1)}^{t-1}k$ is a direct summand of $\syz_R\syz_{R/(x_1)}^{t+1}k$.
For each $n\ge0$ there is an exact sequence $0 \to \syz_{R/(x_1)}^nk \to P_{n-1} \to \cdots \to P_1 \to P_0 \to k \to 0$ with each $P_i$ being a direct sum of copies of $R/(x_1)$, which gives rise to an exact sequence
$$
0 \to \syz_R\syz_{R/(x_1)}^nk \to \syz_RP_{n-1}\oplus R^{\oplus e_{n-1}} \to \cdots \to \syz_RP_1\oplus R^{\oplus e_1} \to \syz_RP_0\oplus R^{\oplus e_0} \to \syz_Rk \to 0
$$
with $e_i\ge0$ for $0\le i\le n-1$.
Note that each $\syz_RP_i$ is a free $R$-module.
The above sequence shows that $\syz_R^{n+1}k=\syz_R^n(\syz_Rk)$ is isomorphic to $\syz_R\syz_{R/(x_1)}^nk$ up to free $R$-summands.
We obtain an $R$-isomorphism $\syz_R^{n+1}k\oplus R^{\oplus e}\cong\syz_R\syz_{R/(x_1)}^nk$ with $e\ge0$.
Thus, for some $a,b\ge0$ we have that $\syz_R^tk\oplus R^{\oplus a}$ is a direct summand of $\syz_R^{t+2}k\oplus R^{\oplus b}$.
Since $R$ is singular, it follows from \cite[Corollary 1.3]{D} that $\syz^i_Rk$ has no nonzero free summand for all $i\ge0$.
Applying Lemma \ref{r11}, we observe that $\syz_R^tk$ is a direct summand of $\syz_R^{t+2}k$.
\end{proof}

We pose a question asking whether or not the converse of Proposition \ref{p46} holds true.

\begin{ques}
Does there exist a non-Burch local ring $(R,\m,k)$ of depth $t$ such that $\syz^tk$ is a direct summand of $\syz^{t+2}k$?
\end{ques}


\section{Some classes of Burch ideals and rings}\label{main3Sec}

In this section, we study in a regular local ring and give a complete characterization in dimension two. We also give a simple characterization of monomial Burch ideals. We compare Burch rings to other classes of rings: radical cube zero, almost Gorenstein, nearly Gorenstein, and fibre products. 

Over a two-dimensional regular local ring $(R,\m)$, the Burch ideals $I$ are characterized in terms of the minimal numbers of generators of $I$ and $\m I$.

\begin{lem} \label{l62}
Let $(R,\m)$ be a regular local ring of dimension two, and let $I$ be an $\m$-primary ideal of $R$.
Then $I$ is a Burch ideal of $R$ if and only if $\mu(\m I)<2\mu(I)$.
\end{lem}

\begin{proof}
It follows from the Hilbert--Burch theorem that $\mu(I)=r(R/I)+1$ and $\mu(\m I)=r(R/\m I)+1$.
The assertion follows from the equivalence (1) $\Leftrightarrow$ (2) in Proposition \ref{23}.
\end{proof}

Now we can show the following theorem, which particularly gives a characterization of the Burch ideals of two-dimensional regular local rings in terms of minimal free resolutions.
Compare this theorem with the result of Corso, Huneke and Vasconcelos \cite[Lemma 3.6]{CHV}.

\begin{thm} \label{t63}
Let $(R,\m)$ be a regular local ring of dimension $d$.
Let $I$ be an $\m$-primary ideal of $R$.
Take a minimal free resolution $0 \to F_d \xrightarrow[]{\varphi_d} F_{d-1} \to \cdots \to F_1 \xrightarrow[]{\varphi_1} F_0 \to R/I \to 0$ of the $R$-module $R/I$.
Consider the following conditions.
\begin{enumerate}[\rm(1)]
\item The ideal $I$ is Burch.
\item There exist a regular system of parameters $x_1,\dots,x_d$ and an integer $r>0$ such that $\I_1(\varphi_d)=(x_1^r,x_2,\dots,x_d)$.
\item One has $(I:\m)^2\not=I(I:\m)$.
\end{enumerate}
Then the implication {\rm(1)} $\Rightarrow$ {\rm(2)} holds.
If $R$ contains a field, then the implication {\rm(3)} $\Rightarrow$ {\rm(2)} holds.
If $d=2$, then the implication {\rm(2)} $\Rightarrow$ {\rm(1)} holds as well.
\end{thm}

\begin{proof}
We first show that (1) implies (2).
We may assume $d\ge 2$, so that $R$ has depth greater than $1$.
By Lemma \ref{31} and its proof, there is a non-zerodivisor $x_1\in \m\setminus \m^2$ such that $I/x_1I$ contains the residue field $R/\m$ as a direct summand.
Tensoring $R/(x)$ with the complex $F=(0\to F_d\to\cdots\to F_0\to0)$, we get a minimal free resolution
$$
(0 \to F_d/x_1F_d \xrightarrow[]{\varphi_d\otimes S/(x_1)} F_{d-1}/F_{d-1} \to \cdots \to F_2/x_1F_2 \to F_1/x_1F_1\to0)
$$
of $I/x_1I$ over $R/(x_1)$.
As $R/\m$ is a direct summand of $I/x_1I$, a minimal $R/(x_1)$-free resolution $G$ of $R/\m$ is a direct summand of the above complex.
Since $G$ is isomorphic to the Koszul complex $\k^{R/(x_1)}$ of $R/(x_1)$, the ideal $\I_1(\varphi_d\otimes R/(x_1))$ of $R/(x_1)$ contains the maximal ideal $\m/(x_1)$.
Therefore $\I_1(\varphi_d)$ contains elements $x_2,\dots,x_d$ such that $x_1,x_2,\dots,x_d$ form a regular system of parameters of $R$.
Since the radical of $\I_1(\varphi_d)$ contains $I$, it is an $\m$-primary ideal.
It follows that there is an integer $r>0$ such that $x_1^r\in \I_1(\varphi_d)$ but $x_1^{r-1}\not\in \I_1(\varphi_d)$.
We obtain $\I_1(\varphi_d)=(x_1^r,x_2,\dots,x_d)$, and (2) follows.

Next, under the assumption that $R$ contains a field, we prove that (3) implies (2).
We use an analogue of the proof of \cite[Theorem 2.4]{CGHPU}.
After completion, we may assume that $R$ is a formal power series ring over a field $k$.
Suppose that (2) does not hold.
Then $d\ge 2$ and we can take an ideal $L$ containing $\I_1(\varphi_d)$ such that there is a regular system of parameters $x_1,\dots,x_d$ with $L=(x_1^2,x_1x_2,x_2^2,x_3,\dots,x_d)$.
By \cite[Proposition 2.1]{CGHPU}, an isomorphism $(I:L)/I\cong\omega_{R/L}\otimes_R F_d$ and its retraction $(I:\m)/I\cong \omega_{R/\m}\otimes_R F_d$ are given.
Note that the canonical module $\omega_{R/L}$ of $R/L$ is isomorphic to $(0:_{\E_R(k)}L)$.
The module $\E_R(k)$ is identified with $k[x_1,x_1^{-1},\dots,x_d^{-1}]/N$, where $N$ is the subspace spaned by the monomials not in $k[x_1^{-1},\dots,x_d^{-1}]$.
Under this identification, $\omega_{R/L}=(0:L)$ is generated by the monomials $x_1^{-1}$ and $x_2^{-1}$.
Set $M=\{x_1^{-1},x_2^{-1}\}$.
Then $x_1M=\{1\}=x_2M$ generates $\omega_{R/\m}$.
Also, either $x_1w=0$ or $x_2w=0$ holds for all $w\in M$.
We may apply \cite[Proposition 2.3]{CGHPU} as in the proof of \cite[Theorem 2.4]{CGHPU} to get $(I:\m)^2=I(I:\m)$, contrary to (3).
We have shown that (3) implies (2).

Finally, assuming $d=2$, we prove (2) implies (1).
As the entries of $\varphi_2$ are contained in $\m$, we have an exact sequence $0 \to F_2 \xrightarrow[]{\varphi_2} \m F_{1} \to \m I \to 0$.
This induces an exact sequence $F_2/\m F_2 \xrightarrow[]{\varphi_2\otimes_R R/\m} \m F_1/\m^2 F_1 \to \m I/\m^2 I \to 0$.
Suppose that (2) holds.
Then $\varphi_2\otimes_R R/\m\not=0$, and $\dim_{R/\m} (\m I/\m^2 I)< \dim_{R/\m} (\m F_1/\m^2 F_1)$.
Note that $\dim_{R/\m} (\m I/\m^2 I)=\mu(I)$ and $\dim_{R/\m} (\m F_1/\m^2 F_1)=2\mu(I)$.
Lemma \ref{l62} shows that $I$ is a Burch ideal, that is, (1) holds.
\end{proof}

\begin{ex}
\begin{enumerate}[(1)]
\item
Let $I=(x^4,y^4,z^4,x^2y,y^2z,z^2x)$ be an ideal of $(R,\m)=k[\![x,y,z]\!]$.
Then one can check that $(I:\m)=(x^4,x^3z,x^2y,xy^3,xyz,xz^2,y^4,y^2z,yz^3,z^4)$, and so $(I:\m)^2\not=I(I:\m)$.
However, $I$ is not Burch.
This gives a counterexample of the implication (3) $\Rightarrow$ (1) in Theorem \ref{t63}.
\item
Let $I=(x^4,y^4,x^3y,xy^3)$ be an ideal of $(R,\m)=k[\![x,y]\!]$.
Then $(I:\m)=(x^3,x^2y^2,y^3)$.
We see that $(I:\m)^2=I(I:\m)$ and $I$ is Burch.
This shows that the implication (1) $\Rightarrow$ (3) in Theorem \ref{t63} is not affirmative, even when $R$ has dimension two.
\end{enumerate}
\end{ex}

We provide some characterizations of Burchness for monomial ideals of regular local rings.

\begin{prop}\label{4}
Let $(R,\m)$ be a regular local ring of dimension $d$.
Let $x_1,\dots,x_d$ be a regular system of parameters of $R$, and let $I$ be a monomial ideal (in the $x_i$s) of $R$.
Then $I$ is Burch if and only if there exist a monomial $m\in I\setminus \m I$ and an integer $1\le i\le d$ such that $x_i\mid m$ and $m(x_j/x_i) \in I$ for all $1\le j\le d$.
\end{prop}

\begin{proof}
Since $I$ is a Burch ideal, we have $\m I\not=\m (I:\m)$.
Therefore, there is a monomial $m'\in (I:\m)$ and an integer $i$ such that $x_im'\not\in \m I$.
It also holds that $x_jm'\in I$ for all $j=1,\dots,d$.
So the element $m:=x_im'$ satisfies $m(x_j/x_i) \in I$ for all $j=1,\dots,d$.
\end{proof}

\begin{cor}
Let $(R,\m)$ be a regular local ring of dimension $2$ with a regular system of parameters $x,y$.
Let $I=(x^{a_1}y^{b_1},x^{a_2}y^{b_2},\dots,x^{a_n}y^{b_n})$ be a monomial ideal with $a_1>a_2>\cdots>a_n$ and $b_1<b_2<\cdots<b_n$.
Then $I$ is a Burch ideal of $R$ if and only if $a_i=a_{i+1}+1$ or $b_i=b_{i+1}-1$ for some $i=1,\dots,n$.
\end{cor}

\begin{proof}
By Proposition \ref{4}, the ideal $I$ is Burch if and only if $x^{a_i}y^{b_i}(y/x)\in I$ or $x^{a_i}y^{b_i}(x/y)\in I$ for some $i=1,\dots, n$.
Equivalently, either $x^{a_i-1}y^{b_i+1}\in I$ or $x^{a_i+1}y^{b_i-1}\in I$ holds for some $i=1,\dots, n$.
Since $a_{i+1}\le a_i-1<a_i<a_i+1\le a_{i-1}$ and $b_{i-1}\le b_i-1<b_i<b_i+1\le b_{i+1}$, the condition is equivalent to saying that $b_i+1=b_{i+1}$ or $a_i+1=a_{i-1}$ for some $i=1,\dots, n$.
\end{proof}

Next, we discuss the relationship between Burch rings and several classes of rings studied previously in the literature. 

Recall that the {\em trace ideal} $\tr M$ of an $R$-module $M$ is defined by $\tr M=\sum_{f\in \Hom_R(M,R)}\im f$.
The following notions are introduced in \cite{HHS,SV}.

\begin{dfn}[Herzog--Hibi--Stamate]
Let $(R,\m)$ be a Cohen--Macaulay local ring with canonical module $\omega$.
Then $R$ is called {\em nearly Gorenstein} if $\tr\omega$ contains $\m$.
\end{dfn}

\begin{dfn}[Striuli--Vraciu]
Let $(R,\m)$ be an artinian local ring.
Then $R$ is called {\em almost Gorenstein}\footnote{There is another notion of an almost Gorenstein ring; see \cite{almgor}.} if $(0:(0:I))\subseteq(I:\m)$ for all ideals $I$ of $R$.
\end{dfn}

It follows from \cite[Proposition 1.1]{HV} that artinian nearly Gorenstein local rings are almost Gorenstein.

We want to consider the relationship of Burchness with near Gorensteinness and almost Gorensteinness.
For this, we establish two lemmas.

\begin{lem}\label{26}
Let $(R,\m,k)$ be a non-Gorenstein artinian almost Gorenstein local ring.
Let $R^n \xrightarrow{A}R^m \to E \to 0$ be a minimal $R$-free presentation of the $R$-module $E=\E_R(k)$.
One then has $\I_1(A)=\m$.
\end{lem}

\begin{proof}
Choose an artinian Gorenstein local ring $(S,\n)$ and an ideal $I$ of $S$ such that $R\cong S/I$.
We identify $E$ with $(0:_SI)$ via the isomorphisms $E\cong\Hom_S(R,S)\cong(0:_SI)$.
Let $x_1,\dots,x_m$ be a minimal system of generators of $E$.
By \cite[Lemma 1.2]{SV} we have $\n=((x_1):_S(x_2,\dots,x_m))+((x_2,\dots,x_m):_Sx_1)$.
We find a matrix $B$ over $S$ with $m$ rows such that $\I_1(B)=\n$ and $\left(\begin{smallmatrix}
x_1&\cdots&x_m
\end{smallmatrix}\right)B=0$.
We find a matrix $C$ over $R$ such that the matrix $\overline B$ over $R$ corresponding to $B$ is equal to $AC$.
We have $\m=\I_1(\overline B)=\I_1(A\cdot C)\subseteq\I_1(A)\subseteq\m$, which implies $\I_1(A)=\m$.
\end{proof}

\begin{lem} \label{l68}
Let $(R,\m)$ be a regular local ring of dimension $d$, and let $I\subseteq \m^2$ be an ideal of $R$.
Take a minimal free resolution $0 \to F_d \xrightarrow{\varphi_d} F_{d-1} \to \cdots \to F_1 \xrightarrow{\varphi_1} F_0 \to R/I \to 0$ of the $R$-module $R/I$.
If $R/I$ is artinian, non-Gorenstein and almost Gorenstein, then $\I_1(\varphi_d)=\m$.
\end{lem}

\begin{proof}
Set $A=R/I$ and $E=\E_A(k)$.
Then the sequence $(F_{d-1}/IF_{d-1})^\ast \xrightarrow[]{(\varphi_d\otimes A)^\ast} (F_d/IF_d)^\ast \to E \to 0$ gives a minimal $A$-free presentation of $E$, where $(-)^\ast=\Hom_A(-,A)$.
Note that $\rank_A (F_d/IF_d)^\ast=r(A)=\mu(E)$.
Lemma \ref{26} implies $\I_1((\varphi_d\otimes A)^\ast)=\m$, which shows $\I_1(\varphi_d)+I=\m$.
The desired result follows from Nakayama's lemma.
\end{proof}

We can show an artinian almost Gorenstein local ring of embedding dimension two is Burch.

\begin{prop} \label{c68}
Let $(R,\m)$ be a regular local ring of dimension $2$ and $I$ an ideal of $R$.
Assume that $R/I$ is a non-Gorenstein artinian almost Gorenstein ring.
Then $I$ is a Burch ideal of $R$.
\end{prop}

\begin{proof}
Take a minimal free resolution $0 \to F_2 \xrightarrow{\varphi_2} F_1 \xrightarrow{\varphi_1} F_0 \to R/I \to 0$ of the $R$-module $R/I$.
It follows from Lemma \ref{l68} that $I_1(\varphi_2)=\m$.
Since $R$ has dimension two, we can use the implication (2)$\Rightarrow$(1) in Theorem \ref{t63} to have that $I$ is Burch.
\end{proof}

\begin{rem}\label{3}
One may hope a non-Gorenstein nearly Gorenstein local ring is Burch, but this is not necessarily true.
Indeed, let $(R,\m)$ be a $1$-dimensional nearly Gorenstein local ring (e.g. $R=k[\![t^3,t^4,t^5]\!]\subseteq k[\![t]\!]$ with $k$ a field).
Take a regular element $x\in\m^2$, and set $A=R/(x)$.
Then $A$ is nearly Gorenstein by \cite[Proposition 2.3(b)]{HHS}, but $A$ is not a Burch ring by Corollary \ref{r4}.
\end{rem}

Next, we deal with local rings the cube of whose maximal ideal is zero.
The following gives a characterization of Burchness for such rings.

\begin{thm} \label{cu}
Let $(R,\m,k)$ be a local ring with $\m^3=0$.
Then $R$ is a Burch ring if and only if there is an inequality $\beta_2(k)>(\edim R)^2-r(R)$.
\end{thm}

\begin{proof}
Put $e=\edim R$ and $r=r(R)$.
By Theorem \ref{210}, the ring $R$ is Burch if and only if $k$ is a direct summand of $\syz^2 k$, if and only if $\soc \syz^2 k\not\subseteq \m\syz^2 k$.
There is a short exact sequence $0 \to \syz^2 k \to R^e \to \m \to 0$, which gives an inclusion $\syz^2 k\subseteq \m R^e$ and an equality $\soc \syz^2 k=\soc R^e$.
Since $\m^3=0$, we have an inclusion $\m\syz^2 k\subseteq \soc \syz^2 k$.
Thus $R$ is Burch if and only if $\ell(\soc \syz^2k)> \ell(\m\syz^2 k)$.
There are equalities
\begin{align*}
\beta_2(k)
&=\ell(\syz^2 k)-\ell(\m\syz^2 k)=\ell(R^e)-\ell(\m)-\ell(\m\syz^2 k)
=(e-1)\ell(\m)+e-\ell(\m\syz^2 k)\\
&=(e-1)(e+\ell(\m^2))+e-\ell(\m\syz^2 k)
=e^2+(e-1)\ell(\m^2)-\ell(\m\syz^2 k).
\end{align*}
On the other hand, there is an inclusion $\syz^2 k\subseteq \m^e$, which induces an inclusion $\m\syz^2 k\subseteq (\m^2)^e$.
Thus one has $\ell(\m\syz^2 k)\le e\ell(\m^2)\le er=\ell(\soc \syz^2k)$.
If $\ell(\m^2)<\ell(\soc R)=r$, then we see that $\ell(\soc \syz^2k)> \ell(\m\syz^2 k)$.
The above equalities show that $\beta_2(k)\ge e^2-\ell(\m^2)>e^2-r$.
Therefore, we may assume $\ell(\m^2)=r$.
We obtain $\beta_2(k)=e^2-r+er-\ell(\m\syz^2 k)$.
It follows that $\beta_2>e^2-r$ if and only if $er-\ell(\m\syz^2 k)>0$.
The latter condition is equivalent to $\ell(\soc \syz^2k)> \ell(\m\syz^2 k)$.
\end{proof}

\begin{cor}
Let $(R,\m,k)$ be a local ring with $\m^3=0$ and $\soc R\subseteq \m^2$.
If $R$ is a Burch ring, then $R$ has no Conca generator in the sense of \cite{AIS}.
\end{cor}

\begin{proof}
If $R$ has a Conca generator, then the Poincar\'{e} series $P_k(t)=\sum \beta_it^i$ is of the form $\frac{1}{1-et+rt^2}$ by \cite[Theorem 1.1]{AIS}.
In particular, $\beta_2(k)=e^2-r$.
Thus $R$ is not Burch by Theorem \ref{cu}.
\end{proof}

Next, we study Burch rings which are fibre products. Let $k$ be a field.
A local ring $R$ is said to be a {\em fibre product} (over $k$) provided that it is of the form
$$
R\cong S\times_kT=\{(s,t)\in S\times T\mid \pi_S(s)=\pi_T(t)\},
$$
where $(S,\m_S)$ and $(T,\m_T)$ are local rings with common residue field $k$, and $\pi_S:S\to k$ and $\pi_T:T\to k$ are the natural surjections.
The fibre product is called {\em nontrivial} if $S\not=k\not=T$.
The set $S\times_k T$ is a local ring with maximal ideal $\m_{S\times_k T}=\m_S\oplus \m_T$ and residue field $k$.
Conversely, a local ring $R$ with decomposable maximal ideal $\m_R=I\oplus J$ is a fibre product since $R\cong (R/I)\times_k (R/J)$.
These observations are due to Ogoma \cite[Lemma 3.1]{Og}.
It holds that $\depth S\times_k T=\min\{\depth S,\depth T,1\}$; see \cite[Remarque 3.3]{L}.

We consider the Burchness of the fibre product $S\times_k T$.
We compute some invariants.

\begin{lem} \label{lfp}
Let $R=S\times_k T$ be a nontrivial fibre product, where $(S,\m_S,k)$ and $(T,\m_T,k)$ are local rings.
Then the following equalities hold.
\begin{enumerate}[\rm(1)]
\item 
$\edim R=\edim S+\edim T$.
\item
$\dim_k \soc R=\dim_k \soc S+\dim_k \soc T$.
\item
$\dim_k \h_1(\k^R)=\dim_k \h_1(\k^S)+\dim_k \h_1(\k^T)+\edim S\cdot\edim T$.
\item
$c_R=c_S+c_T+\edim S\cdot\edim T-\edim(S/\soc S)\cdot\edim(T/\soc T)$.
\end{enumerate}
\end{lem}

\begin{proof}
(1)(2) These equalities can be checked directly.

(3) One has $\beta_2^R(k)=\beta_2^S(k)+\beta_2^T(k)+2\edim S\cdot\edim T$ and $\dim_k \h_1(\k^R)=\beta_2^R(k)-\binom{\edim R}{2}$; see \cite{KS} and \cite[Theorem 2.3.2]{BH} for example.
Thus there are equalities
\begin{align*}
\dim_k \h_1(\k^R) &=\beta_2^R(k)-\tbinom{\edim R}{2}
=\beta_2^S(k)+\beta_2^T(k)+2\edim S\cdot\edim T-\tbinom{\edim R}{2}\\
&=\dim_k \h_1(\k^S)-\tbinom{\edim S}{2}+\dim_k \h_1(\k^T)-\tbinom{\edim T}{2}+2\edim S\cdot\edim T-\tbinom{\edim R}{2}\\
&=\dim_k \h_1(\k^{R_1})+\dim_k \h_1(\k^{R_2})+\edim S\cdot \edim T.
\end{align*} 

(4) Put $R'=R/\soc R$, $S'=S/\soc S$ and $T'=T/\soc T$.
Then $R'\cong S'\times T'$ unless $S=k$ or $T=k$.
Using (1), (2) and (3), we can calculate $c_R$ as follows:
\begin{align*}
c_R &=\dim_k\soc R+\dim \h_1(\k^R)-\edim R-\dim \h_1(\k^{R'})+\edim R'\\
&=\dim_k\soc S+\dim_k\soc T+\dim_k \h_1(\k^S)+\dim_k \h_1(\k^{R_2})+\edim S\cdot \edim T\\
&-\edim S-\edim T-\dim_k\h_1(\k^{S'})-\dim_k \h_1(\k^{T'})-\edim S'\cdot\edim T'+\edim S'+\edim T'\\
&=c_S+c_T+\edim S\cdot\edim T-\edim S'\cdot \edim T'.\qedhere
\end{align*}
\end{proof}

Using the above lemma, we can characterize the Burch fibre products.

\begin{prop} \label{fp}
Let $R=S\times_k T$ be a nontrivial fibre product, where $(S,\m_S,k)$ and $(T,\m_T,k)$ are local rings.
Then $R$ is a Burch ring if and only if\\
\qquad{\rm(a)} $\depth R>0$,\quad
or\quad
{\rm(b)} $\depth R=0$ and either $S$ or $T$ is a Burch ring of depth zero.
\end{prop}

\begin{proof}
First we deal with the case where $\depth R=0$.
Lemma \ref{l310} shows that $R$ is Burch if and only if $c_R>0$.
Note that the integers $c_S, c_T$ and $N:=\edim S\cdot\edim T-\edim(S/\soc S)\cdot\edim(T/\soc T)$ are always nonnegative.
By Lemmas \ref{lfp}(4), the positivity of $c_S$ or $c_T$ implies that $R$ is Burch.
Conversely, assume that $R$ is Burch.
Then by Lemma \ref{lfp}(4) again, one of the three integers $c_S$, $c_T,N$ is positive.
If $c_S$ or $c_T$ is positive, then $S$ or $T$ is Burch.
When $N>0$, either $\edim S>\edim S/\soc S$ or $\edim T>\edim T/\soc T$ holds.
Without loss of generality, we may assume that $\edim S>\edim S/\soc S$.
This inequality means that there is an element $x\in(\m_S\cap\soc S)\setminus \m_S^2$.
Then $\m_S=I\oplus (x)$ for some ideal $I$.
We see that $S\cong S/(x)\times_k S/I$ and $\edim S/I\le 1$.
Example \ref{e310} implies that $S/I$ is Burch, and so is $S$.

Next, we consider the case where $\depth R>0$.
In this case, we have $\depth S>0$, $\depth T>0$ and $\depth R=1$.
Take regular elements $x\in\m_S\setminus \m_S^2$ and $y\in\m_T\setminus \m_T^2$.
The element $x-y\in\m_R=\m_S\oplus\m_T$ is also a regular element of $R$.
The equalities $x\m_R=x\m_S=(x-y)\m_S$ show that the image $\overline{x}\in R/(x-y)$ of $x$ is in $\soc R/(x-y)$.
We have $\m_R/(x-y)=(\overline{x})\oplus I$ for some ideal $I$ of $R/(x-y)$.
Hence $R/(x-y)$ is isomorphic to the fibre product $U\times_kV$ of local rings over their common residue field $k$ such that $\edim V\le1$.
As $V$ is Burch by Example \ref{e310}, it follows that so is $R/(x-y)$, and hence so is $R$.
\end{proof}

\begin{ex}\label{r2}
Let $R=k[x,y]/(x^a,xy,y^b)$ with $k$ a field and $a,b\ge1$.
Then $R$ is a Burch ring.
In fact, $R$ is isomorphic to the fibre product of $k[x]/(x^a)$ and $k[y]/(y^b)$ over $k$.
By Example \ref{e310}, the rings $k[x]/(x^a)$ and $k[y]/(y^b)$ are Burch, and so is $R$ by Proposition \ref{fp}.
\end{ex}

\section{Homological and categorical properties of Burch rings}\label{main4Sec}

In this section, we explore some homological and categorical aspects of Burch rings. They come in several flavors. 
We prove a classification theorem of subcategories over Burch rings.
We also prove that non-Gorenstein Burch rings are G-regular in the sense of \cite{greg}, and that nontrivial consecutive vanishings of Tor over Burch rings cannot happen.
We begin with recalling the definition of resolving subcategories.

\begin{dfn}
Let $R$ be a ring.
A subcategory $\X$ of $\mod R$ is {\em resolving} ifthe following hold.
\begin{enumerate}[(1)]
\item
The projective $R$-modules belong to $\X$.
\item
Let $M$ be an $R$-module and $N$ a direct summand of $M$.
If $M$ is in $\X$, then so is $N$.
\item
For an exact sequence $0\to L\to M\to N\to0$, if $L$ and $N$ are in $\X$, then so is $M$.
\item
For an exact sequence $0\to L\to M\to N\to0$, if $M$ and $N$ are in $\X$, then so is $L$.
\end{enumerate}

Note that (1) can be replaced by the condition that $\X$ contains $R$.
Also, (4) can be replaced by the condition that if $M$ is an $R$-module in $\X$, then so is $\syz M$.
For an $R$-module $C$, we denote by $\res_RC$ the {\em resolving closure} of $C$, the smallest resolving subcategory of $\mod R$ containing $C$.
\end{dfn}

We establish a couple of lemmas to prove Proposition \ref{r1}.
The first lemma is used as a base result of this section, which is essentially shown in \cite[Proposition 4.2]{res3r}.
For an $R$-module $M$ we denote by $\nf(M)$ the {\em nonfree locus} of $M$, that is, the set of prime ideals $\p$ of $R$ such that $M_\p$ is nonfree as an $R_\p$-module.

\begin{lem} \label{l52}
Let $(R,\m)$ be a local ring, $M$ a nonfree $R$-module, and $x$ an element in $\m$.
\begin{enumerate}[\rm(1)]
\item
There exists a short exact sequence $0\to \syz M \to M(x) \to M \to 0$ such that $x\in\I_1(M(x))\subseteq\m$ and $\pd_RM(x)\ge \pd_R M$.
In particular, $M(x)$ belongs to $\res_R M$.
\item
For each $\p\in\V(x)\cap\nf(M)$ one has $\V(\p)\subseteq \nf(M(x))\subseteq \nf(M)$ and $\D(x)\cap \nf(M(x))=\emptyset$.
\end{enumerate}
\end{lem}

\begin{proof}
(1) Let $\cdots\xrightarrow{d_3} F_2\xrightarrow{d_2} F_1 \xrightarrow{d_1} F_0\xrightarrow{\pi}M\to0$ be a minimal free resolution of $M$.
Taking the mapping cone of the multiplication map of the complex $F$ by $x$, we get an exact sequence
$$
\cdots\to F_3\oplus F_2\xrightarrow{\left(\begin{smallmatrix}d_3&x\\0&-d_2\end{smallmatrix}\right)}F_2\oplus F_1\xrightarrow{\left(\begin{smallmatrix}d_2&x\\0&-d_1\end{smallmatrix}\right)}F_1\oplus F_0\xrightarrow{\left(\begin{smallmatrix}d_1&x\\0&-\pi\end{smallmatrix}\right)}F_0\oplus M\xrightarrow{\left(\begin{smallmatrix}\pi&x\end{smallmatrix}\right)}M\to0.
$$
Set $M(x)=\im\left(\begin{smallmatrix}d_1&x\\0&-\pi\end{smallmatrix}\right)=\cok\left(\begin{smallmatrix}d_2&x\\0&-d_1\end{smallmatrix}\right)$.
The free resolution of $M(x)$ given by truncating the above sequence is minimal.
We see that $x\in\I_1(M(x))\subseteq\m$ as $M$ is nonfree, and that $\pd_RM(x)\ge\pd_R M$.
The following pullback diagram gives an exact sequence as in the assertion.
$$
\xymatrix{
0\ar[r] & \syz M\ar[r]^f & F_0\ar[r]^\pi & M\ar[r] & 0\\
0\ar[r] & \syz M\ar[r]\ar@{=}[u] & M(x)\ar[r]\ar[u] & M\ar[r]\ar[u]_{x} & 0
}
$$

(2) The module $M(x)$ fits into the pushout diagram
$$
\xymatrix{
0\ar[r] & \syz M\ar[r]^f\ar[d]^{x} & F_0\ar[r]^\pi\ar[d] & M\ar[r]\ar@{=}[d] & 0\phantom{.}\\
0\ar[r] & \syz M\ar[r] & M(x)\ar[r] & M\ar[r] & 0.
}
$$
Using the same argument as in the proof of \cite[Proposition 4.2]{res3r}, we observe that $\V(\p)\subseteq \nf(M(x))\subseteq \nf(M)$ and $\D(x)\cap\nf(M(x))=\emptyset$ hold.
\end{proof}

\begin{lem} \label{l53}
Let $(R,\m)$ be a local ring and $M$ an $R$-module.
Let $W\subseteq \nf(M)$ be a closed subset of $\spec R$.
Then there exists an $R$-module $X$ such that $\pd_R X\ge\pd_R M$ and $\nf(X)=W$.
\end{lem}

\begin{proof}
The assertion follows from the proof of \cite[Theorem 4.3]{res3r} by replacing \cite[Lemma 4.2]{res3r} used there with our Lemma \ref{l52}.
\end{proof}

\begin{lem} \label{301}
Let $(R,\m)$ be a local ring and $M$ a nonfree $R$-module.
Then there is an exact sequence $0\to (\syz M)^n \to N \to M^n \to 0$ with $n\ge1$, $\I_1(N)=\m$ and $\pd_RN\ge\pd_RM$.
In particular, $N\in\res_R M$.
\end{lem}

\begin{proof}
Let $x_1,\dots,x_n$ be a minimal system of generators of $\m$.
According to Lemma \ref{l52}, for each $i$ there exists an exact sequence $0\to\syz M\to M(x_i)\to M\to0$ such that $x_i\in\I_1(M(x_i))\subseteq\m$ and $\pd_R M(x_i)\ge\pd_RM$.
Putting $N=\bigoplus_{i=1}^nM(x_i)$, we obtain an exact sequence $0 \to (\syz M)^n \to N \to M^n \to 0$ with $\I_1(N)=\sum_{i=1}^n\I_1(M(x_i))=\m$ and $\pd_RN\ge\pd_RM$.
\end{proof}

\begin{lem} \label{l54}
Let $R$ be a local ring.
Let $M$ be an $R$-module that is locally free on the punctured spectrum of $R$.
\begin{enumerate}[\rm(1)]
\item 
For every $X\in \res_{\widehat{R}}\widehat{M}$ there exists $Y\in\res_RM$ such that $X$ is a direct summand of $\widehat{Y}$.
\item 
Let $N$ be an $R$-module.
If $\widehat{N}\in \res_{\widehat{R}}\widehat{M}$, then $N\in \res_R M$.
\end{enumerate}
\end{lem}

\begin{proof}
(1) Let $\C$ be the subcategory of $\mod\widehat R$ consisting of direct summands of the completions of modules in $\res_RM$.
We claim that $\C$ is a resolving subcategory of $\mod\widehat R$ containing $\widehat M$.
Indeed, since $R,M$ are in $\res_RM$, the completions $\widehat R,\widehat M$ are in $\C$.
For each $E\in\C$, there exists $D\in\res_RM$ such that $E$ is a direct summand of $\widehat D$.
The module $\syz_{\widehat R}E$ is a direct summand of $\syz_{\widehat R}\widehat D=\widehat{\syz_RD}$.
As $\syz_RD\in\res_RM$, we have $\syz_RE\in\C$.
Let $0\to A\to B\to C\to0$ be an exact sequence of $\widehat R$-modules with $A,C\in\C$.
Then $A,C$ are direct summands of $\widehat V,\widehat W$ for some $V,W\in\res_RM$, respectively.
Writing $A\oplus A'=\widehat V$ and $C\oplus C'=\widehat W$, we get an exact sequence $\sigma:0\to \widehat V\to B'\to \widehat W\to0$, where $B'=A'\oplus B\oplus C'$.
The exact sequence $\sigma$ corresponds to an element of $\Ext_{\widehat R}^1(\widehat W,\widehat V)=\widehat{\Ext_R^1(W,V)}$.
Since $M$ is locally free on the punctured spectrum of $R$, so are $V$ and $W$.
Hence $\Ext_R^1(W,V)$ has finite length as an $R$-module, and is complete.
This implies that there exists an exact sequence $\tau:0\to V\to U\to W\to0$ of $R$-modules such that $\widehat\tau\cong\sigma$.
Thereofre $U$ is in $\res_RM$ and $B'$ is isomorphic to $\widehat U$.
Thus $B$ belongs to $\C$, and the claim follows.
The claim shows that $\C$ contains $\res_{\widehat R}\widehat M$.
Hence $X$ is in $\C$, which shows the assertion.

(2) By (1) there is an $R$-module $Y\in \res_R M$ such that $\widehat{N}$ is a direct summand of $\widehat{Y}$.
Thanks to \cite[Corollary 1.15(i)]{LW}, the module $N$ is a direct summand of $Y$.
Hence $N$ belongs to $\res_R M$.
\end{proof}

Now we can show the proposition below, which yields a significant property of Burch rings.
This is also used in the proofs of Theorem \ref{33} and \ref{th54}.

\begin{prop}\label{r1}
Let $R$ be a Burch local ring of depth $t$ with residue field $k$.
Let $M$ be an $R$-module of infinite projective dimension.
Then $\syz^tk$ belongs to $\res_RM$.
\end{prop}

\begin{proof}
We begin with proving the proposition when $R$ is complete and $t=0$.
As $M$ has infinite projective dimension, Lemma \ref{301} gives rise to an $R$-module $N\in\res_RM$ with $\I_1(N)=\m$.
Proposition \ref{203} implies that $k$ is a direct summand of $\syz^2_RN$.
As $\syz^2_RN$ is in $\res_RM$, so is $k$.

Now, let us consider the case where $R$ is complete and $t>0$.
By definition, there is a maximal regular sequence $\xx$ of $R$ such that $R/(\xx)$ is a Burch ring of depth $0$.
Note that $\syz^tM\in \res_R M$.
For all $i>0$ we have $\Tor_i^R(\syz^t M,R/(\xx))=\Tor_{i+t}^R(M,R/(\xx))=0$, which says that $\xx$ is a regular sequence on $\syz^t M$.
The $R/(\xx)$-module $\syz^t M/\xx\syz^t M$ has infinite projective dimension by \cite[Lemma 1.3.5]{BH}.
The case $t=0$ implies that $k$ belongs to $\res_{R/(\xx)}\syz^t M/\xx\syz^t M$.
It follows from \cite[Lemma 5.8]{stcm} that $\syz_R^dk\in\res_R\syz^t M\subseteq \res_R M$.

Finally, we consider the case where $R$ is not complete.
Lemma \ref{l53} gives an $R$-module $X\in \res_RM$ with $\pd_R X=\infty$ and $\nf(X)=\{\m\}$.
As $\widehat{R}$ is Burch and $\pd_{\widehat{R}}\widehat{X}=\pd_R X=\infty$, the above argument yields $\syz^t_{\widehat{R}} k\in\res_{\widehat{R}} \widehat{X}$.
Using Lemma \ref{l54}, we see $\syz^t k\in\res_R X$, and $\syz^tk\in\res_R M$.
\end{proof}

Non-Gorenstein Burch rings admit only trivial totally reflexive modules.
Recall that a local ring $R$ is called {\em G-regular} if every totally reflexive $R$-module is free.

\begin{thm}\label{33}
Let $R$ be a non-Gorenstein Burch local ring.
Then $R$ is G-regular. 
\end{thm}

\begin{proof}
By taking the completion and using \cite[Corollary 4.7]{greg}, we may assume that $R$ is complete.
Let $\mathcal{G}$ be the category of totally reflexive $R$-modules.
Then $\mathcal{G}$ is a resolving subcategory of $\mod R$ by \cite[(1.1.10) and (1.1.11)]{C}.
If $R$ is not G-regular, that is, there is a nonfree $R$-module $M$ in $\mathcal{G}$,
then Proposition \ref{r1} shows that $\mathcal{G}$ contains the $R$-module $\syz^d k$, where $d=\dim R$.
In other words, $\syz^d k$ is totally reflexive.
This especially says that the $R$-module $k$ has finite G-dimension, and $R$ is Gorenstein; see \cite[(1.4.9)]{C}.
This contradiction shows that $R$ is G-regular.
\end{proof}

\begin{rem}
The converse of Theorem \ref{33} does not necessarily hold.
In fact, the non-trivial fibre product $R=S\times_k T$ of non-Burch local rings $S,T$ is non-Burch.
However, thanks to \cite[Lemma 4.4]{NT}, the same argument of the proof of Theorem \ref{33} works, and hence $R$ is G-regular.
\end{rem}

As a corollary of Theorem \ref{33}, ``embedded deformations'' of Burch rings are never Burch.

\begin{cor}\label{r4}
Let $(R,\m)$ be a singular local ring.
Let $x\in\m^2$ be an $R$-regular element.
Then the local ring $R/(x)$ is not Burch.
\end{cor}

\begin{proof}
The proof of \cite[Proposition 4.6]{greg} gives rise to an endomorphism $\delta:R^n\to R^n$ such that $\delta^2=x\cdot\id_{R^n}$ and $\im\delta\subseteq\m R^n$.
It is easy to see that $\delta$ is injective, and we have an exact sequence $0\to R^n\xrightarrow{\delta}R^n\to C\to0$ with $xC=0$.
This induces an exact sequence $\cdots\xrightarrow{\overline\delta}(R/(x))^n\xrightarrow{\overline\delta}(R/(x))^n\xrightarrow{\overline\delta}(R/(x))^n\xrightarrow{\overline\delta}\cdots$ of $R/(x)$-modules whose $R/(x)$-dual is exact as well.
Since $\im\overline\delta=C$, the $R/(x)$-module $C$ is totally reflexive.
As $\im\delta\subseteq\m R^n$, we see that $C$ is not $R/(x)$-free.
Hence $R/(x)$ is not G-regular.

Suppose that $R/(x)$ is Burch.
Then Theorem \ref{33} implies that $R/(x)$ is Gorenstein.
By Proposition \ref{r5}, the ring $R/(x)$ is a hypersurface.
We have
$$
1\ge\codepth R/(x)=\edim R/(x)-\depth R/(x)=\edim R-(\dim R-1)=\codim R+1,
$$
where the second equality follows from the assumption that $x$ is not in $\m^2$.
We get $\codim R=0$, which means that $R$ is regular, contrary to our assumption.
\end{proof}

Let $(R,\m)$ be a local ring.
We denote by $\speco R$ the punctured spectrum of $R$.
For a property $\P$, we say that {\em $\speco R$ satisfies $\P$} if $R_\p$ satisfies $\P$ for all $\p\in\speco R$.
We denote by $\cm(R)$ the subcategory of $\mod R$ consisting of maximal Cohen--Macaulay modules.
Also, $\D^\b(R)$ stands for the bounded derived category of $\mod R$, and $\D_\sg(R)$ the {\em singularity category} of $R$, that is, the Verdier quotient of $\D^\b(R)$ by perfect complexes.
Note that $\D^\b(R)$ and $\D_\sg(R)$ have the structure of a triangulated category.
A {\em thick} subcategory of a triangulated category is by definition a triangulated subcategory closed under direct summands.
The following theorem gives rise to classifications of several kinds of subcategories over Burch rings; recall that a Cohen--Macaulay local ring $R$ is said to have {\em finite Cohen--Macaulay representation type} if there exist only finitely many isomorphism classes of indecomposable maximal Cohen--Macaulay $R$-modules.
For the unexplained notations and terminologies appearing in the theorem, we refer to \cite[\S2]{NT}.

\begin{thm} \label{th54}
Let $(R,\m)$ be a singular Cohen--Macaulay Burch local ring.
\begin{enumerate}[\rm(1)]
\item
Suppose that $\speco R$ is either a hypersurface or has minimal multiplicity.
Then there is a commutative diagram of mutually inverse bijections:
$$
\xymatrix{
{\left\{
\begin{matrix}
\text{Resolving subcategories of}\\
\text{$\mod R$ contained in $\cm(R)$}
\end{matrix}
\right\}}
\ar@<.7mm>[r]^-\nf\ar@{=}[d] &
{\left\{
\begin{matrix}
\text{Specialization-closed}\\
\text{subsets of $\sing R$}
\end{matrix}
\right\}}
\ar@<.7mm>[l]^-{\nf^{-1}_\cm}\ar@<.7mm>[d]^-{\ipd^{-1}} \\
{\left\{
\begin{matrix}
\text{Thick subcategories of}\\
\text{$\cm(R)$ containing $R$}
\end{matrix}
\right\}}
\ar@<.7mm>[r]^-{\thick_{\mod R}} \ar@<.7mm>[d]^-{\thick_{\D_{\sg}(R)}}&
{\left\{
\begin{matrix}
\text{Thick subcategories of}\\
\text{$\mod R$ containing $R$}
\end{matrix}
\right\}}
\ar@<.7mm>[l]^-{\rest_{\cm(R)}}\ar@<.7mm>[u]^-\ipd \ar@<.7mm>[d]^-{\thick_{\D^{\b}(R)}}
\\
{\left\{
\begin{matrix}
\text{Thick subcategories of}\\
\text{$\D_{\sg}(R)$}
\end{matrix}
\right\}}
\ar@<.7mm>[r]^-{\pi^{-1}}  \ar@<.7mm>[u]^-{\rest_{\cm(R)}}&
{\left\{
\begin{matrix}
\text{Thick subcategories of}\\
\text{$\D^{\b}(R)$ containing $R$}
\end{matrix}
\right\}}
\ar@<.7mm>[l]^-{\pi}\ar@<.7mm>[u]^-{\rest_{\mod R}}
}
$$
\item
Assume that $R$ is excellent and admits a canonical module $\omega$.
Suppose that $\speco R$ has finite Cohen--Macaulay representation type.
Then there is a commutative diagram of mutually inverse bijections:
$$
\xymatrix{
{\left\{
\begin{matrix}
\text{Resolving subcategories}\\
\text{of $\mod R$ contained in}\\
\text{$\cm(R)$ and containing $\omega$}
\end{matrix}
\right\}}
\ar@<.7mm>[r]^-\nf\ar@{=}[d] &
{\left\{
\begin{matrix}
\text{Specialization-closed}\\
\text{subsets of $\sing R$}\\
\text{containing $\ng R$}
\end{matrix}
\right\}}
\ar@<.7mm>[l]^-{\nf^{-1}_\cm}\ar@<.7mm>[d]^-{\ipd^{-1}} \\
{\left\{
\begin{matrix}
\text{Thick subcategories of}\\
\text{$\cm(R)$ containing}\\
\text{$R$ and $\omega$}
\end{matrix}
\right\}}
\ar@<.7mm>[r]^-{\thick_{\mod R}} \ar@<.7mm>[d]^-{\thick_{\D_{\sg}(R)}}&
{\left\{
\begin{matrix}
\text{Thick subcategories of}\\
\text{$\mod R$ containing}\\
\text{$R$ and $\omega$}
\end{matrix}
\right\}}
\ar@<.7mm>[l]^-{\rest_{\cm(R)}}\ar@<.7mm>[u]^-\ipd \ar@<.7mm>[d]^-{\thick_{\D^{\b}(R)}}\\
{\left\{
\begin{matrix}
\text{Thick subcategories of}\\
\text{$\D_{\sg}(R)$ containing $\omega$}
\end{matrix}
\right\}}
\ar@<.7mm>[r]^-{\pi^{-1}}  \ar@<.7mm>[u]^-{\rest_{\cm(R)}}&
{\left\{
\begin{matrix}
\text{Thick subcategories of}\\
\text{$\D^{\b}(R)$ containing}\\
\text{$R$ and $\omega$}
\end{matrix}
\right\}}
\ar@<.7mm>[l]^-{\pi}\ar@<.7mm>[u]^-{\rest_{\mod R}}
}
$$
\end{enumerate}
\end{thm}

\begin{proof}
The proof of \cite[Theorem 4.5]{NT} uses \cite[Lemma 4.4]{NT}.
Replace this lemma with our Proposition \ref{r1}.
Then the same argument works, and the theorem follows.
\end{proof}

\begin{ex}
We have the following list of examples of non-Gorenstein Cohen--Macaulay local rings not having isolated singularities, where $\circ$ and $\times$ mean ``Yes'' and ``No'' respectively.
\\
\\
\begin{tabular}{c|c|c|c|c|c|c}
\hline
&  &  &  & \multicolumn{3}{c}{$\speco R$}\\
\cline{5-7}
\cite[Example]{crs} & $R$ & $\dim R$ & Burch & hypersurface & min. mult. & finite CM rep. type\\
\hline\hline
7.1 & $\displaystyle\frac{k[\![x,y,z]\!]}{(x^2,xz,yz)}$ & $1$ & $\circ$ & $\circ$ & $\circ$ & $\circ$ \\
\hline
7.2 & $\displaystyle\frac{k[\![x,y,z]\!]}{(x^2,xy,y^2)}$ & $1$ & $\circ$ & $\times$ & $\circ$ & $\times$ \\
\hline
7.3 & $\displaystyle\frac{k[\![x,y,z]\!]}{(xy,z^2,zw,w^3)}$ & $1$ & $\times$ & $\times$ & $\circ$ & $\times$ \\
\hline
7.4 & $\displaystyle\frac{k[\![x,y,z]\!]}{(x^2-yz,xy,y^2)}$ & $1$ & $\circ$ & $\circ$ & $\times$ & $\circ$ \\
\hline
7.5 & $\displaystyle\frac{k[\![x,y,z,w]\!]}{(xy,xz,yz)}$ & $2$ & $\circ$ & $\times$ & $\circ$ & $\circ$ \\
\hline
\end{tabular}
\\
\\
The assertions are shown in \cite[Examples 7.1--7.5]{crs}, except those on the Burch property.
As to the first, second, fourth and fifth rings $R$ are Burch since the quotient of a system of parameters is isomorphic to $k[x,y]/(x^2,xy,y^2)$, which is an artinian Burch ring by Example \ref{r2}.
As for the third ring $R$, note that $(x,y)$ is an exact pair of zerodivisors.
Hence it is not G-regular, and not Burch by Theorem \ref{33}.
\end{ex}

Now we discuss the vanishing of Tor modules over Burch rings.
The following result is a simple consequence of Lemmas \ref{l310} and \ref{301}.

\begin{prop} \label{l59}
Let $(R,\m,k)$ be a Burch ring of depth zero, and let $M,N$ be $R$-modules.
If $\Tor_l^R(M,N)=\Tor_{l+1}^R(M,N)=0$ for some $l\ge 3$, then either $M$ or $N$ is a free $R$-module.
\end{prop}

\begin{proof}
We may assume that $R$ is complete.
Assume that $M$ is nonfree.
Since $\depth R=0$, the $R$-module $M$ has infinite projective dimension.
By Lemma \ref{301}, there is a short exact sequence $0\to (\syz M)^n \to X \to M^n \to 0$, where $X$ satisfies $\I_1(X)=\m$.
It induces an exact sequence $0\to (\syz^3 M)^n \to \syz^2 X\oplus F \to (\syz^2 M)^n \to 0$ with $F$ free.
We also have $\Tor_{l-2}(\syz^2 M,N)=\Tor_{l-2}(\syz^3 M,N)=0$, which implies that $\Tor_{l-2}(\syz^2 X,N)=0$.
Lemma \ref{l310} implies that $k$ is a direct summand of $\syz^2 X$, as $R$ is Burch.
We see that $\Tor_{l-2}(k,N)$ vanishes.
This shows that $N$ has finite projective dimension, and so it is $R$-free.
\end{proof}

We can prove the following by applying a similar argument as in the proof of \cite[Corollary 6.5]{NT}, where we use Proposition \ref{l59} instead of \cite[Corollary 6.2]{NT}.

\begin{cor}\label{1}
Let $(R,\m,k)$ be a Burch ring of depth $t$.
Let $M,N$ be $R$-modules.
Assume that there exists an integer $l\ge\max\{3,t+1\}$ such that $\Tor_i^R(M,N)=0$ for all $l+t\le i\le l+2t+1$.
Then either $M$ or $N$ has finite projective dimension.
\end{cor}

\begin{rem}
Using an analogous argument as in the proof of \cite[Corollary 6.6]{NT}, one can also prove a variant of Corollary \ref{1} regarding Ext modules.
\end{rem}

We state a remark on the ascent of Burchness along a flat local homomorphism.

\begin{rem}
Let $(R,\m)\to(S,\n)$ be a flat local homomorphism of local rings.
Even if the rings $R$ and $S/\m S$ are Burch, $S$ is not necessarily Burch.
In fact, consider the natural injection
$$
\phi:R=k[x,y]/(x^2,xy,y^2)\hookrightarrow k[x,y,t]/(x^2,xy,y^2,t^2)=S.
$$
Then $\phi$ is a flat local homomorphism.
The artinian local rings $R$ and $S/\m S=k[t]/(t^2)$ are Burch by Examples \ref{r2} and \ref{r3}(1).
The ring $S$ is not G-regular since $(t,t)$ is an exact pair of zerodivisors of $S$.
Theorem \ref{33} implies that $S$ is not Burch.
\end{rem}

\section*{Acknowledgments}
Most of this work was done during the visit of Toshinori Kobayashi to the University of Kansas in 2018--2019. 
He is grateful to the Department of Mathematics for their hospitality.
We thank Professors Rodney Sharp and Edmund Robertson for providing them with useful biographical data on Professor Burch. We also thank Craig Huneke and Shinya Kumashiro for helpful comments. 
Hailong Dao was partly supported by Simons Collaboration Grant FND0077558.
Toshinori Kobayashi was partly supported by JSPS Grant-in-Aid for JSPS Fellows 18J20660.
Ryo Takahashi was partly supported by JSPS Grant-in-Aid for Scientific Research 16K05098, 19K03443 and JSPS Fund for the Promotion of Joint International Research 16KK0099.

\end{document}